\documentclass[a4paper,10pt]{article}

\pdfoutput=1
\renewcommand{\tilde}{\widetilde}

\usepackage{geometry}
\usepackage{graphicx}
\usepackage{tikz}
\usepackage{comment}
\usepackage{color}
\usepackage{pgfplots}
\usepackage{pgfplotstable}
\usepackage{booktabs}
\usepackage{listings}
\usepackage{amsmath,amssymb,amsfonts,amsthm,cite}
\usetikzlibrary{patterns,pgfplots.groupplots}
\usepackage{enumerate}
\usepackage{algorithm}
\usepackage{algpseudocode}
\usepackage{stmaryrd}
\usepackage{mathrsfs}
\usepackage{hyperref}

\renewcommand{\tilde}{\widetilde}

\newcommand{\rank}{\mathrm{rank}}
\newcommand{\Span}{\mathrm{span}}
\newcommand{\T}[1]{\mathcal {#1}}

\colorlet{DenseBlockColor}{gray!60}

\date{}

\pgfplotstableset{
	every head row/.style={before row=\toprule,after row=\midrule},
	clear infinite
}
\newtheorem{definition}{Definition}[section]

\newtheorem{remark}{Remark}

\newtheorem{lemma}[definition]{Lemma}
\newtheorem{theorem}[definition]{Theorem}
\newtheorem{corollary}[definition]{Corollary}

\newcommand{\norm}[1]{\lVert#1\rVert}
\definecolor{mygreen}{RGB}{28,172,0} 
\definecolor{mylilas}{RGB}{170,55,241}
\definecolor{stringcolor}{RGB}{180,10,10}
\definecolor{mygray}{RGB}{240,240,240}
\definecolor{mygray2}{RGB}{200,200,200}
\author{
	Stefano Massei\footnote{Department of Mathematics, University of Pisa. 
		E-mails: stefano.massei@unipi.it, leonardo.robol@unipi.it. Both authors are members of the INdAM/GNCS research group. The research of Leonardo Robol
		has been partially supported by the italian national research center in
		HPC, Big Data and Quantum Computing.} 
	\and
	Leonardo Robol\footnotemark[\value{footnote}]
}

\title{A nested divide-and-conquer method for tensor Sylvester equations with 
	positive definite hierarchically semiseparable coefficients}

\pgfplotsset{compat=1.15}

\begin{document}
\maketitle

\begin{abstract}
	Linear systems with a tensor product structure arise naturally when considering the discretization of Laplace type differential equations or, more generally, multidimensional operators with separable coefficients.  
	In this work, we focus on the numerical solution of linear systems of the form
	$$ \left(I\otimes \dots\otimes I \otimes  A_1+\dots + A_d\otimes I \otimes\dots \otimes I\right)x=b,
	$$
	where the matrices $A_t\in\mathbb R^{n\times n}$ are symmetric positive definite and belong to the class of hierarchically semiseparable matrices. 
	
	We propose and analyze a nested  divide-and-conquer scheme, based on the technology of low-rank updates, that attains the quasi-optimal computational cost $\mathcal O(n^d\log(n))$. Our theoretical analysis highlights the role of inexactness in the nested calls of our algorithm and provides worst case estimates for the amplification of the residual norm. The performances are validated on 2D and 3D case studies.  
	
	\end{abstract}


	
	\section{Introduction}
	
	In this work we are concerned with the numerical solution of linear systems with a Kronecker sum-structured coefficient matrix of the form:
	\begin{equation}\label{eq:lin-sys}
	\left(I\otimes\dots  \otimes I\otimes A_1+\dots + A_d\otimes I \otimes \dots \otimes I\right)x=b,
	\end{equation}
	where the matrices $A_t\in\mathbb R^{n_t\times n_t}$ are symmetric and positive definite
	(SPD) with spectrum contained in $[\alpha_t, \beta_t]\subset \mathbb R^+:=\{z\in\mathbb R\ \text{s.t.}\ z>0\}$, 
	and have low-rank off-diagonal blocks, for $t=1,\dots,d$. 
	By reshaping $x,b\in\mathbb R^{n_1\dots n_d}$  into $d$-dimensional tensors $\T X,\T{B}\in\mathbb R^{n_1\times \dots n_d}$, we rewrite \eqref{eq:lin-sys} as the tensor Sylvester equation
	\begin{equation}\label{eq:tens-sylv}
	\T X\times_1 A_1+\dots+\T X\times_d A_d=\T B,
	\end{equation}
	where $\times_j$ denotes the $j$-mode product for tensors \cite[Section 2.6]{kolda}. 
	
	Tensor Sylvester equations, not necessarily with SPD coefficients, arise naturally when discretizing $d$-dimensional
	Laplace-like operators by means 
	of tensorized grids that respect the separability of the operator
	\cite{townsend2015automatic,townsend-fortunato,strossner2021fast,palitta2016matrix,massei2021rational}. 
	In the case 
	$d = 2$, we recover the well-known case of matrix Sylvester equations, that also 
	plays a dominant role in the model reduction of dynamical systems 
	\cite{antoulas-approximation}. 
	
	Several methods for solving matrix and tensor Sylvester equations assume
	that the right-hand side has some structure, such as being low-rank. This is a necessary 
	assumption for dealing with large scale problems. In this paper, we only consider unstructured right-hand sides $\T B$, for which the cases of interest are those where the memory cost $\mathcal O(\prod_{t=1}^d n_t)$ still allows storing $\T B$ and the solution $\T X$. Note that, this limits the scenarios where our  algorithm is effective to small values of $d$, i.e. $d=2,3$. 
	The structure in
	the coefficients $A_t$ (which are SPD and have off-diagonal blocks of low-rank), 
	will be crucial to improve the complexity of the solver with respect to the 
	completely unstructured case. 
	We also remark that when the $A_t$s 
	arise from the discretization of elliptic differential operators, 
	the structure assumed in this work is often present 
	\cite{hackbusch,borm}.
	
	\subsection{Related work}
	In the matrix case (i.e., $d = 2$ in \eqref{eq:tens-sylv}) there are two main 
	procedures that make no assumptions on $A_1, A_2,$ and $\T B$: the Bartels-Stewart algorithm~\cite{bartels-stewart,recursive-sylvester-kagstrom} and 
	the Hessenberg-Schur method~\cite{hessenberg-schur}. These are based on taking the coefficients $A_1,A_2$ 
	to either Hessenberg or triangular form, and then solving the linear system by (block) 
	back-substitution. The idea has been generalized to $d$-dimensional tensor 
	Sylvester equations in \cite{chen2020recursive}. In the case 
	where $n = n_1 = \ldots = n_d$, the computational complexity of these 
	approaches is $\mathcal O(dn^3 + n^{d+1})$ flops. 
	
	When $d=2$ and the right-hand side $\T B$ is a low-rank matrix or $d>2$ and $\T B$ is representable in a low-rank 
	tensor format (Tucker, Tensor-Train, Hierarchical Tucker, \ldots) the tensor equation 
	can be solved much more efficiently, and the returned approximate solution is low-rank, 
	which allows us to store it in a low-memory format. Indeed, in this case it is possible 
	to exploit tensorized Krylov (and rational Krylov) methods \cite{Simoncini2016,druskin2011analysis,kressner2009krylov}, or the 
	factored Alternating Direction 
	Implicit Method (fADI) \cite{benner2009adi,shi21}. The latter methods build a rank $s$ approximant 
	to the solution $\T X$ by solving  $\mathcal O(s)$ 
	shifted linear systems with the matrices $A_t$. This is very effective when 
	also the coefficients $A_t$ are structured. For instance, when the $A_t$ 
	are sparse or hierarchically low-rank, this often brings the cost of approximating $X$ 
	to $\mathcal O(sn \log^an)$ for $a \in \{0,1,2\}$ \cite{hackbusch}. In the tensor 
	case, another option is to rely on methods adapted to the low-rank structure 
	under consideration: 
	AMEn \cite{dolgov2014alternating} or TT-GMRES \cite{dolgov2013tt} 
	for Tensor-Trains, projection methods in the 
	Hierarchical Tucker format \cite{ballani2013projection}, and other approaches.
	
	In this work, we consider an intermediate setting, where the coefficients $A_t$
	are structured, while the right-hand side $\T B$ is not. More specifically, we assume 
	that the $A_t$ are SPD and efficiently representable in 
	the Hierarchical SemiSeparable format (HSS) \cite{xia2010fast}. This implies that 
	each coefficient $A_t$ can be partitioned in a $2 \times 2$ block matrix 
	with low-rank off-diagonal blocks, and diagonal blocks with the same recursive structure. 
	
	A particular case of this setting has been considered in 
	\cite{townsend-fortunato}, where the $A_t$
	are banded SPD (and therefore have low-rank off-diagonal blocks), and a nested 
	Alternating Direction Implicit (ADI) solver is applied to a 3D tensor equation 
	with no structure in $\T B$. The complexity of the algorithm is quasi-optimal 
	$\mathcal O(n^3 \log^3 n)$, but the hidden constant is very large, and the approach is 
	not practical already for moderate $n$; see \cite{strossner2021fast} for a comparison 
	with methods with a higher asymptotic complexity. 
	
	We remark that, when the coefficients $A_t$ are SPD, the tensor equation \eqref{eq:tens-sylv} can be solved by diagonalization of the $A_t$s
	in a stable way, as described in the pseudocode of Algorithm~\ref{alg:diag}. Without further assumptions, this costs $\mathcal O(dn^3 + n^{d+1})$
	when all dimensions are equal.  
	\begin{algorithm}[H] 
		\small 
		\caption{}\label{alg:diag}
		\begin{algorithmic}[1]
			\Procedure{lyapnd\_diag}{$A_1,A_2,\dots, A_d,\T B$}
			\For{$i=1,\dots,d$}
					\State $[S_i, D_i]=\texttt{eig}(A_i)$
							\State $\T B \gets \T B \times_i S_i^*$
				\EndFor	
			\For{$i_1=1,\dots, n_1,\dots, i_d=1,\dots,n_d$}
			\State $\T X(i_1,\dots,i_d)\gets \T B(i_1,\dots,i_d) /  ([D_1]_{i_1i_1}+\dots + [D_d]_{i_di_d})$
			\EndFor 
			\For{$i=1,\dots,d$}
			\State $\T X \gets \T X \times_i S_i$
			\EndFor	
			\State\Return $\T X$
			\EndProcedure
		\end{algorithmic}
	\end{algorithm} 
	
	If the $A_t$ can be efficiently diagonalized, 
	then Algorithm~\ref{alg:diag} attains a quasi-optimal complexity. For instance, 
	in the case of finite difference discretizations of the $d$-dimensional Laplace operator, 
	diagonalizing the matrices $A_t$ via the fast sine or cosine transforms (depending on the boundary conditions) yields the complexity 
	$\mathcal O(n^d \log n)$. 
	Recently, it has been shown that positive definite HSS enjoy a structured eigendecomposition 
	\cite{ou2022superdc}, that can be retrieved in $\mathcal O(n \log^2 n)$ time. In addition, multiplying a vector by the eigenvector matrix costs only $\mathcal O(n\log(n))$ because the latter can be represented as a product of permutations and Cauchy-like matrices, of logarithmic length. These features
	can be exploited into Algorithm~\ref{alg:diag} to obtain an efficient solver. 
	The approach proposed in this work has the same $\mathcal O(n^d \log n)$ 
	asymptotic complexity but, 
	as we will demonstrate in our numerical experiments, 
	will result in significantly lower computational costs. 
	
	\subsection{Main contributions}
	
	The main contribution is the design and analysis of an algorithm 
	with $\mathcal O(\prod_{j = 1}^d n_j \log(\max_{j} n_j))$
	complexity 
	for
	solving the tensor equation \eqref{eq:tens-sylv} with HSS and SPD coefficients 
	$A_t$. The algorithm is based on a divide-and-conquer scheme, where \eqref{eq:tens-sylv}
	is decomposed into several tensor equations that have either a
	low-rank right-hand side, or a small dimension. In the tensor case, the low-rank 
	equations are solved exploiting nested calls to the $(d-1)$-dimensional solver. 
	Concerning the theoretical analysis, we provide the following contributions:
	\begin{itemize}
		\item An error analysis that, assuming a bounded residual error on the low-rank 
		  subproblems, guarantees the accuracy on the final result of the divide-and-conquer scheme
		  (Theorem~\ref{thm:2d-accuracy} and Lemma~\ref{lem:residual-d>2}). 
		\item A novel a priori error analysis for the use of fADI with inexact solves; 
		  more precisely, in Theorem~\ref{thm:adi-res-inex} we provide an explicit bound for the difference 
		  between the residual norm after $s$ steps of fADI in exact arithmetic, and the one obtained 
		  by performing the fADI steps with inexact solves. This enables us to control the residual 
		  norm of the error based only on the number of shifts used in all calls to fADI 
		  in our solver (Theorem~\ref{thm:accuracy-d}). 
		  These results are very much related to those in \cite[Theorem 3.4 and Corollary 3.1]{kurschner2020inexact}, where also the convergence of fADI with inexact solves is analyzed. Nevertheless, the assumptions and the techniques used in the proofs of such results are quite different. The goal of  \cite{kurschner2020inexact} is to progressively  increase the level of  inexactness, along the iterations of fADI, ensuring that the final residual norm remain below a target threshold. The authors proposed an adaptive relaxation strategy that requires the computation of intermediate quantities generated during the execution of the algorithm. In our work, the level of inexactness is fixed and, by exploiting the decay of the residual norm when using optimal shifts, we provide upper bounds for the number of iterations needed to attain the target accuracy.
		\item We prove that for a $d$-dimensional problem, the condition number $\kappa$ of the 
		  tensor Sylvester equation 
		  can (in principle) amplify the residual norm by a factor $\kappa^{d - 1}$, when a
		  nested solver is used. 
		  When the $A_t$ are $M$-matrices, we show that the impact is reduced to 
		  $(\sqrt{\kappa})^{d - 1}$ (Lemma~\ref{lem:m-matrices}). 
		\item A thorough complexity analysis (Theorem~\ref{thm:2d-complexity} and 
		  \ref{thm:3d-complexity}), where the role of the HSS ranks, the target accuracy, 
		  and the 
		  condition number of the $A_t$s are fully revealed. In particular, we show 
		  that the condition numbers have a mild impact on the computational cost.  
	\end{itemize}
	The paper is organized as follows. In Section~\ref{sec:high-level}, we provide a 
	high-level description of the proposed scheme, for a $d$-dimensional tensor Sylvester 
	equation. Section~\ref{sec:2d-case} and Section~\ref{sec:tensors} are
	dedicated to the theoretical analysis of 
	the algorithm for the matrix and tensor case, respectively. Finally, 
	in the numerical experiments of Section~\ref{sec:numerical-experiments} 
	we compare the proposed algorithm with Algorithm~\ref{alg:diag}
	where the diagonalization is performed with a dense method or with the
	algorithm proposed in \cite{ou2022superdc} for HSS matrices. 
	\subsection{Notation}
	Throughout the paper, we denote matrices with capital 
	letters ($X$, $Y$, \ldots), and tensors with 
	calligraphic letters 
	($\T X, \T Y$, \ldots). We use the same letter with different 
	fonts to denote matricizations of tensors (e.g., $X$ is a matricization of 
	$\T X$).
	The Greek letters $\alpha_t, \beta_t$ indicate the 
	extrema of the interval $[\alpha_t, \beta_t]$ enclosing the spectrum 
	of $A_t$, and $\kappa$ is used to denote the upper bound on 
	the condition number of the 
	Sylvester operator 
	$\kappa = (\beta_1 + \ldots + \beta_d) / (\alpha_1 + \ldots + \alpha_d)$. 
	
	\section{High-level description of the divide-and-conquer scheme}
	\label{sec:high-level}
	
	We consider HSS matrices $A_t$, so  that each $A_t$ 
	can be decomposed as $A_t=A_t^{\mathrm{diag}} + A_t^{\mathrm{off}}$ 
	where $A_t^{\mathrm{diag}}$ is block diagonal with square diagonal blocks, 
	$A_t^{\mathrm{off}}$ is low-rank and
	the decomposition applies recursively to the blocks of
	$A_t^{\mathrm{diag}}$. 
	A particular case, where this assumption is satisfied,  is when the coefficients $A_t$ 
	are all banded.
	
	In the spirit of divide and conquer solvers for matrix equations~\cite{kressner2019low,massei2022hierarchical}, we 
	remark that, given the additive decomposition 
	$A_1=A_1^{\mathrm{diag}}+A_1^{\mathrm{off}}$, 
	the solution $\T X$ of \eqref{eq:tens-sylv} can
	 be written as $\T X^{(1)}+\delta \T X$ where
	\begin{align}
	\T X^{(1)}\times_1 A_1^{\mathrm{diag}}+\T X^{(1)}\times_2 A_2+\dots+\T X^{(1)}\times_d A_d&=\T B,\label{eq:bench}\\
	\delta \T X\times_1 A_1+\delta\T X\times_2 A_2+\dots+\delta\T X\times_d A_d&=-\T X^{(1)} \times_1 A_1^{\mathrm{off}}.\label{eq:update}
	\end{align} 
	If $A_1^{\mathrm{diag}}= \left[\begin{smallmatrix}
	A^{(1)}_{1,11} \\ &A^{(1)}_{1,22}
	\end{smallmatrix}\right]$,
	then \eqref{eq:bench} decouples into two tensor equations of the form
	\begin{equation}\label{eq:diag}
	\T X^{(1)}_{j}\times_1 A_{1,jj}^{(1)}+\T X^{(1)}_{j}\times_2 A_2+\dots+\T X^{(1)}_{j}\times_d A_d=\T B_{j}, \qquad j=1,2,
	\end{equation}  
	with $\T X^{(1)}_{1}$ containing the entries of $\T X^{(1)}$ with the first index restricted to the
	 column indices of $A_{1,11}^{(1)}$, and $\T X^{(1)}_{2}$
	to those of $A_{1,22}^{(1)}$.\footnote{In Matlab notation, if $A_{1,11}^{(1)}$ is of size $m\times m$ we have $\T X_1^{(1)}=\T X^{(1)}(1:m, :, \dots,:)$ and $\T X_2^{(1)}=\T X^{(1)}(m+1:\mathrm{end}, :, \dots,:)$} 
	 Equation~\eqref{eq:update} has the notable property that its 
	 right-hand side is a $d$-dimensional tensor multiplied in the first mode by a low-rank matrix. 
	 Merging the modes from $2$ to $d$ (in the sense of \cite[Section 2.6]{kolda}) in \eqref{eq:update} yields the matrix Sylvester equation
	\begin{equation}\label{eq:reshape-update}
	A_1 \delta X + \delta X \left( I\otimes\dots I\otimes A_2 +\dots +A_d\otimes I\otimes\dots\otimes I\right)= -A_1^{\mathrm{off}}X^{(1)}.
	\end{equation} 
	In particular, the right-hand side of \eqref{eq:reshape-update} 
	has rank bounded by $\rank(A_1^{\mathrm{off}})$ and the matrix 
	coefficients of the equation are positive definite. This implies 
	that $\delta X$ has numerically low-rank and can be efficiently 
	approximated with a low-rank Sylvester solver such as a rational 
	Krylov  subspace method \cite{Simoncini2016,druskin2011analysis} 
	or the \emph{alternating direction implicit method} (ADI) 
	\cite{benner2009adi}. 
	
	We note that, applying the splitting simultaneously to all $d$ 
	modes yields an update equation for $\delta \T X$ of the form 
	\begin{equation} \label{eq:update-all}
		\delta \T X\times_1 A_1+\delta\T X\times_2 A_2+\dots+\delta\T X\times_d A_d 
		=- \sum_{t = 1}^d  \T X^{(1)} \times_t A_t^{\mathrm{off}}, 
	\end{equation}
	and $2^d$ recursive calls:
	\begin{equation} \label{eq:diag-all}
		\T X^{(1)}_{j_1, \ldots, j_d} \times_1 A_{1,j_1 j_1}^{(1)} + 
		\dots +
		\T X^{(1)}_{j_1, \ldots, j_d} 
		\times_d A_{d, j_d j_d} = 
		\T B_{j_1, \ldots, j_d}, \qquad j_t \in \{ 1,2 \}. 
	\end{equation}
	However, when $d > 2$, the right-hand side of \eqref{eq:update-all} is not 
	necessarily low-rank for any 
	matricization. On the other hand, by additive splitting 
	of the right-hand side we can write 
	$\delta \T X := \delta \T X_1 + \ldots + \delta \T X_d$, where
	$\delta \T X_t$ is the solution to an equation of the form 
	\eqref{eq:update}. 
	
	In view of the above discussion, we propose the following recursive strategy for solving \eqref{eq:lin-sys}:
	\begin{enumerate}
		\item if all the $n_i$s are sufficiently small 
		  then solve \eqref{eq:lin-sys} by diagonalization,
		\item otherwise split
		 the equation along all modes as in 
		 \eqref{eq:update-all} and \eqref{eq:diag-all},
		\item compute $\T X^{(1)}$ by solving the $2^d$ equations in \eqref{eq:diag-all} recursively,
		\item approximate $\delta \T X_t$ 
		  by applying a low-rank matrix Sylvester solver 
		  for $t = 1, \ldots, d$. 
		\item return $\T X = \T X^{(1)}+\delta \T X_1 + \ldots + \delta \T X_d$.
	\end{enumerate} 
	The procedure will be summarized in Algorithm~\ref{alg:dac} of Section~\ref{sec:tensors}, 
	where we will consider the case of tensors in detail. 
	To address point 4.\ we can use any of the available low-rank 
	solvers for Sylvester equations \cite{Simoncini2016}; in this work 
	we consider the fADI and the rational Krylov subspace methods that 
	are discussed in detail in the next sections. In Algorithm~\ref{alg:dac} 
	we refer to the chosen method with \textsc{low\_rank\_sylv}. We remark
	 that both these choices require to have a low-rank factorization of 
	 the mode $j$ unfolding of  $\T X_0\times_j A_j^{\mathrm{off}}$ and to 
	 solve shifted linear systems with a Kronecker sum of $d-1$ matrices $A_t$.
	The latter task is again of the form \eqref{eq:lin-sys} with $d-1$ modes and is performed recursively with Algorithm~\ref{alg:dac}, when $d>2$; this makes our algorithm a nested solver. At the base of the recursion, when \eqref{eq:lin-sys} has only one mode, this is just a shifted linear system. We discuss this in detail in section~\ref{sec:low-rank}.

	\subsection{Notation for Hierarchical matrices}
	\label{sec:hierarchical}
	The HSS matrices $A_t$ ($t = 1, \ldots, d$) can be partitioned as follows:
	\begin{equation} \label{eq:hodlr-split}
		A_t = \begin{bmatrix}
			A_{t, 11}^{(1)} & A_{t, 12}^{(1)} \\ 
			A_{t, 21}^{(1)} & A_{t, 22}^{(1)} \\
		\end{bmatrix} \in \mathbb R^{n_t \times n_t}, 
	\end{equation}
	where $A_{t,12}^{(1)}$ and $A_{t,21}^{(1)}$ have low rank, 
	and $A_{t, ii}^{(1)}$ 
	are HSS matrices. In particular, the diagonal blocks are square and can 
	be recursively partitioned in the same way $\ell_t - 1$ times.
	The depth $\ell_t$ is chosen to ensure that the blocks 
	at the lowest level of the recursion are smaller than 
	a prescribed minimal size $n_{\min} \times n_{\min}$. 
	
	More formally, after one step of recursion 
	we partition $I_1^{(0)} = \{ 1, \ldots, n_t \} = 
	I_1^{(1)} \sqcup I_2^{(1)}$ where $I_1^{(1)}$ and $I_2^{(1)}$ are 
	two sets of contiguous indices; the matrices 
	$A_{t, ij}$ in \eqref{eq:hodlr-split} have 
	$I_i^{(1)}$ and $I_j^{(1)}$ as row and column indices, 
	respectively, for $i,j = 1,2$. 
	
	Similarly, after $h \leq \ell_t$ steps of recursion, one has 
	the partitioning $\{ 1, \ldots, n_t \} = 
	I_1^{(h)} \sqcup \ldots \sqcup I_{2^{h}}^{(h)}$, and we denote 
	by $A_{t,ij}^{(h)}$ with $1 \leq i,j \leq 2^h$ the 
	submatrices of $A_t$ with row indices $I_i^{(h)}$ and 
	column indices $I_j^{(h)}$. Note that $A_{t,ii+1}^{(h)}$ 
	and $A_{t,i+1i}^{(h)}$ indicate the low-rank off-diagonal 
	blocks uncovered at level $h$ (see Figure~\ref{fig:hodlr}).
	
	The quad-tree structure of submatrices of $A_t$, corresponding to the above described splitting of row and column indices, is called the \emph{cluster tree} of $A_t$; see \cite[Definition 1]{massei2022hierarchical} for the rigorous definition. The integer $\ell_t$ is called the \emph{depth of the cluster tree}.  
	
	\begin{figure}
		\centering 
		\begin{tikzpicture}[scale=0.8] \small
			\coordinate (T18) at (0,0);  
			\coordinate (T14) at (-4,-1);
			\coordinate (T58) at (4,-1);
			\coordinate (T12) at (-6,-2);
			\coordinate (T34) at (-2,-2);
			\coordinate (T56) at (2,-2);
			\coordinate (T78) at (6,-2);
			\coordinate (T1) at (-7,-3);
			\coordinate (T2) at (-5,-3);
			\coordinate (T3) at (-3,-3);
			\coordinate (T4) at (-1,-3);
			\coordinate (T5) at (1,-3);
			\coordinate (T6) at (3,-3);
			\coordinate (T7) at (5,-3);
			\coordinate (T8) at (7,-3);
			\node (N18) at (T18) {$I^{(0)}_1$};
			\node (N14) at (T14) {$I_1^{(1)}$};
			\node (N58) at (T58) {$I_2^{(1)}$};
			\node (N12) at (T12) {$I_1^{(2)}$};
			\node (N34) at (T34) {$I_2^{(2)}$};
			\node (N56) at (T56) {$I_3^{(2)}$};
			\node (N78) at (T78) {$I_4^{(2)}$};
			\node (N1) at (T1) {$I_1^{(3)}$};
			\node (N2) at (T2) {$I_2^{(3)}$};
			\node (N3) at (T3) {$I_3^{(3)}$};
			\node (N4) at (T4) {$I_4^{(3)}$};
			\node (N5) at (T5) {$I_5^{(3)}$};
			\node (N6) at (T6) {$I_6^{(3)}$};   
			\node (N7) at (T7) {$I_7^{(3)}$};
			\node (N8) at (T8) {$I_8^{(3)}$};
			\draw[->] (N18.south) -- (N14);
			\draw[->] (N18.south) -- (N58);  
			\draw[->] (N14.south) -- (N12);
			\draw[->] (N14.south) -- (N34);    
			\draw[->] (N58.south) -- (N56);  
			\draw[->] (N58.south) -- (N78);
			\draw[->] (N12.south) -- (N1);
			\draw[->] (N12.south) -- (N2);
			\draw[->] (N34.south) -- (N3);  
			\draw[->] (N34.south) -- (N4);  
			\draw[->] (N56.south) -- (N5);  
			\draw[->] (N56.south) -- (N6);          
			\draw[->] (N78.south) -- (N7);  
			\draw[->] (N78.south) -- (N8);          
		\end{tikzpicture} \\[.5cm]
		\begin{tikzpicture}[scale=0.35]    
			  \begin{scope}
				[xshift=0cm]
				\node [above] at (4,8) {$h=0$};
				\draw (0,0) -- (0,8) -- (8,8) -- (8,0) -- cycle;
				\node at (4,4) {$A_t$};
			  \end{scope}
			  \begin{scope}
				[xshift=9cm]
				\node [above] at (4,8) {$h=1$};
				\draw (0,0) -- (0,8) -- (8,8) -- (8,0) -- cycle;
				\draw[fill=blue!25] (0,0) rectangle (4,4);
				\draw[fill=blue!25] (4,4) rectangle (8,8);
				\draw (0,4) -- (8,4);
				\draw (4,0) -- (4,8);
				\node at (2,6) {$A_{t,11}^{(1)}$};
				\node at (2,2) {$A_{t,21}^{(1)}$};
				\node at (6,6) {$A_{t,12}^{(1)}$};
				\node at (6,2) {$A_{t,22}^{(1)}$};
			  \end{scope}
			  \begin{scope}
				[xshift=18cm]
				\node [above] at (4,8) {$h=2$};
				\draw (0,0) -- (0,8) -- (8,8) -- (8,0) -- cycle;
				\draw[fill=blue!25] (4,0) rectangle (6,2);
				\draw[fill=blue!25] (0,4) rectangle (2,6);
				\draw[fill=blue!25] (6,2) rectangle (8,4);
				\draw[fill=blue!25] (2,6) rectangle (4,8);
				\draw (0,2) -- (8,2);
				\draw (0,4) -- (8,4);
				\draw (0,6) -- (8,6);
				\draw (2,0) -- (2,8);
				\draw (4,0) -- (4,8);
				\draw (6,0) -- (6,8);
				\foreach \i in {1, ..., 4} {
					\foreach \j in {1, ..., 4} {
						\node at (2*\i-1, 8-2*\j+1) {\tiny $A_{t,\j\i}^{(h)}$};
					}
				}
			  \end{scope}
			  \begin{scope}
				[xshift=27cm]
				\node [above] at (4,8) {$h=3$};
				\draw (0,0) -- (0,8) -- (8,8) -- (8,0) -- cycle;
				\draw[fill=blue!25] (6,0) rectangle (7,1);
				\draw[fill=blue!25] (4,2) rectangle (5,3);
				\draw[fill=blue!25] (2,4) rectangle (3,5);
				\draw[fill=blue!25] (0,6) rectangle (1,7);
				\draw[fill=blue!25] (7,1) rectangle (8,2);
				\draw[fill=blue!25] (5,3) rectangle (6,4);
				\draw[fill=blue!25] (3,5) rectangle (4,6);
				\draw[fill=blue!25] (1,7) rectangle (2,8);
				
				\draw (0,1) -- (8,1);
				\draw (0,2) -- (8,2);
				\draw (0,3) -- (8,3);
				\draw (0,4) -- (8,4);
				\draw (0,5) -- (8,5);
				\draw (0,6) -- (8,6);
				\draw (0,7) -- (8,7);
				\draw (1,0) -- (1,8);
				\draw (2,0) -- (2,8);
				\draw (3,0) -- (3,8);
				\draw (4,0) -- (4,8);
				\draw (5,0) -- (5,8);
				\draw (6,0) -- (6,8);
				\draw (7,0) -- (7,8);
			  \end{scope}
		\end{tikzpicture}
		\caption{Example of the hierarchical low-rank structure 
		  obtained with the recursive partitioning in 
		  \eqref{eq:hodlr-split}. The light blue blocks are the 
		  low-rank submatrices identified at each level.}
		\label{fig:hodlr}
	\end{figure}
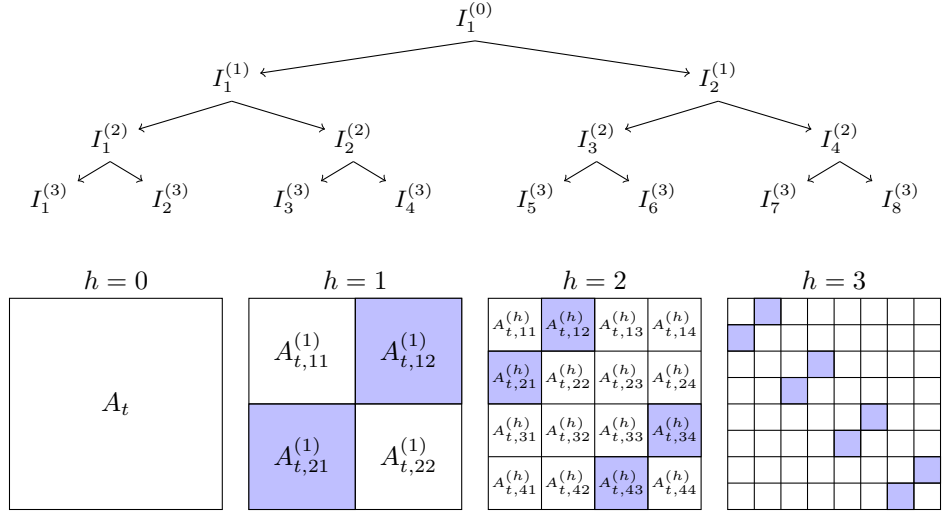
	
	Often, we will need to group together all the diagonal blocks 
	at level $h$; we denote by $A_t^{(h)}$ such matrix, that is:
	\begin{equation} \label{eq:hodlr-level-h}
		A_t = 
		\underbrace{\begin{bmatrix}
			A_{t,11}^{(h)}&  &  &  \\
			 & \ddots &  & \\
			 &  &  \ddots & \\
			 & &  &A_{t,2^{h} 2^{h}}^{(h)}
			\end{bmatrix}}_{A_t^{(h)}} + 
		\begin{bmatrix}
				0 & \star & \dots & \star \\
				\star &\ddots & \ddots & \vdots \\
				\vdots & \ddots & \ddots & \star\\
				\star & \dots & \star & 0
		\end{bmatrix}.
	\end{equation}
	Finally, the maximum rank of the off-diagonal blocks of an HSS matrix is called the \emph{HSS rank}. 
	
	\subsection{Representation and operations with HSS matrices}
	An $n\times n$ matrix in the form described in the previous section, with HSS rank $k$, can be 
	effectively stored in the HSS format \cite{xia2010fast}, using only 
	$\mathcal O(nk)$ memory. Using this structured representation, 
	matrix-vector multiplications and solution 
	of linear systems can be performed with $\mathcal O(nk)$ and 
	$\mathcal O(nk^2)$ flops, respectively. 
	
	Our numerical results leverage the implementation of this format 
	and the related matrix operations 
	available in \texttt{hm-toolbox} \cite{massei2020hm}. 
	
	\section{The divide-and-conquer approach for matrix Sylvester equations}
	\label{sec:2d-case}
	We begin by discussing the case $d=2$, that is 
	the matrix Sylvester equation 
	\begin{equation} \label{eq:sylv2d}
		A_1 X + XA_2 = B, 
		\qquad 
		B \in \mathbb{C}^{n_1 \times n_2},
	\end{equation}
	since there is a
	major difference with respect to $d>2$ that makes the 
	theoretical analysis much simpler. Indeed, the call 
	to \textsc{low\_rank\_sylv} for solving \eqref{eq:reshape-update}
	does not need to recursively call the divide-and-conquer scheme, 
	which is needed when $d > 2$ for generating the 
	right Krylov subspaces. Moreover, we assume that 
	$\ell_1 = \ell_2 =: \ell$, that automatically implies that 
	$n_1$ and $n_2$ are of the same order of magnitude; the 
	algorithm can be easily adjusted for unbalanced dimensions, 
	as we discuss in detail in Section~\ref{sec:unbalanced-2d}. 
	
	We will denote by 
	$B^{(h)} = [B_{ij}^{(h)}]$ 
	the matrix $B$ seen as a block matrix partitioned according 
	to the cluster trees of $A_1$ and $A_2$ at level $h$ for 
	the rows and columns, respectively. 
	
	Splitting both modes at once yields the update equation
	\begin{equation}\label{eq:2dcase}
	A_1\delta X + \delta X A_2 = -A_1^{\mathrm{off}} X^{(1)} - 
	X^{(1)} A_2^{\mathrm{off}},\qquad X^{(1)}=\begin{bmatrix}
	X^{(1)}_{11}& X_{12}^{(1)} \\ 
	X_{21}^{(1)} & X_{22}^{(1)}
	\end{bmatrix}
	\end{equation}
	where $X_{ij}^{(1)}$ is the solution of
	$A_{1,ii}^{(1)} X_{ij}^{(1)} + 
	X_{ij}^{(1)} A_{2,jj}^{(1)} = B^{(1)}_{ij}$ 
	and $B^{(1)}_{ij}$ is the block in position $(i,j)$ in 
	$B^{(1)}$. 
	The right-hand side of \eqref{eq:2dcase} has rank bounded by 
	$\rank(A_1^{\mathrm{off}})+\rank(A_2^{\mathrm{off}})$, 
	therefore we can use \textsc{low\_rank\_sylv} 
	to solve \eqref{eq:2dcase}.  The procedure is summarized in Algorithm~\ref{alg:dac2d-balanced}. 
	
	\begin{algorithm}[H] 
		\small 
		\caption{}\label{alg:dac2d-balanced}
		\begin{algorithmic}[1] 
			\Procedure{lyap2d\_d\&c\_balanced}{$A_1,A_2, B, \epsilon$}
			\If{$\max_i n_i\leq n_{\min}$}
			\State\Return\Call{lyapnd\_diag}{$A_1,A_2, B$}
			\Else
			
					\State $X_{ij}^{(1)} \gets$ \Call{Lyap2D\_D\&C}{$A_{1,ii}^{(1)}, A_{2,jj}^{(1)}, B_{ij}^{(1)}$}, \ for $i,j = 1,2$
					\State $X^{(1)} \gets \left[
						\begin{smallmatrix}
							X_{11}^{(1)} & X_{12}^{(1)} \\ 
							X_{21}^{(1)} & X_{22}^{(1)} \\ 
						\end{smallmatrix}
					\right]$
					\State Retrieve a low-rank factorization of 
				  $A_1^{\mathrm{off}} X^{(1)} + X^{(1)} A_2^{\mathrm{off}} = UV^T$\label{step:fact}
			
			\State $\delta X\gets \Call{low\_rank\_sylv}{A_1, A_2, U, V, \epsilon}$
			\State\Return $X^{(1)}+\delta X$
			\EndIf
			\EndProcedure
		\end{algorithmic}
	\end{algorithm} 
	
	\begin{remark}\label{rem:rhs}
		Efficient solvers of Sylvester equations with low-rank right-hand sides
		need a factorization of the latter; see line~\ref{step:fact} 
		in Algorithm~\ref{alg:dac2d-balanced}. Once the matrix $X^{(1)}$ is 
		computed, the factors $U$ and $V$ are retrieved with explicit formula 
		involving the factorizations of the matrices $A_1^{\mathrm{off}}$ and 
		$A_2^{\mathrm{off}}$ \cite[Section 3.1]{kressner2019low}. The low-rank 
		representation can be cheaply compressed via a QR-SVD based procedure; 
		our implementation always apply this compression step. 
	\end{remark}
	
	\subsection{Analysis of the equations generated in the recursion}
	\label{sec:equations-recursion}
	
	In this section we introduce the notation for all the equations 
	solved in the recursion, as this will be useful in the error and 
	complexity analysis. 
	
	We denote with capital letters (e.g., $X, \delta X$)
	the exact solutions of such equations, and with an additional
	tilde 
	(i.e., $\widetilde{X}$ and $\delta \widetilde{X})$ 
	their inexact counterparts obtained in finite 
	precision computations. 
	
	\subsubsection{Exact arithmetic}
	\label{sec:d=2-exact-equations}
	
	We begin by considering the computation, with Algorithm~\ref{alg:dac2d-balanced}, of the solution of 
	a matrix Sylvester equation, assuming that 
	all the equations generated during the recursion are solved exactly. 
	In this scenario, the 
	solution $X$ admits the additive splitting 
	\[
		X = X^{(\ell)} + \delta X^{(\ell-1)} + \ldots + \delta X^{(0)}, 
	\]
	where $X^{(\ell)}$ or $\delta X^{(h)}$ contains all the 
	solutions determined 
	at depth $\ell$ and $h$ in the recursion, respectively. 
	More precisely, $X^{(\ell)}$ 
	takes the form 
	\begin{equation} \label{eq:xsplitting}
		X^{(\ell)} := \begin{bmatrix}
			X_{1,1}^{(\ell)} & \dots & X_{1, 2^{\ell}}^{(\ell)} \\ 
			\vdots    & & \vdots \\
			\T X_{2^{\ell}, 1}^{(\ell)} & \dots & \T X_{2^{\ell},2^{\ell}}^{(\ell)}. 
		\end{bmatrix},
	\end{equation}
	where $X^{(\ell)}_{i,j}$ solves the Sylvester equation 
	$A_{1, ii}^{(\ell)} X^{(\ell)}_{i,j} + 
	X^{(\ell)}_{i,j} A_{2,jj}^{(\ell)} = B_{i,j}$; for 
	$h < \ell$, we denote $X^{(h)} := X^{(\ell)} + 
	\delta X^{(\ell-1)} + \ldots + \delta X^{(h)}$, that solves $A_1^{(h)} X^{(h)} + X^{(h)} A_2^{(h)} = B^{(h)}$.

	The matrix $\delta X^{(h)} = [\delta X^{(h)}_{i,j}]$ 
	containing the solutions of the
	update equations at level $h < \ell$ is block-partitioned
	analogously 
	to $B^{(h)}$ and $X^{(h)}$. 
	The diagonal blocks $A_{1, ii}^{(h)}$ can be 
	in turn 
	split into their diagonal and off-diagonal parts as follows:
	\[
		A_{1, ii}^{(h)} = \begin{bmatrix}
			A_{1, 2i-1, 2i-1}^{(h + 1)} \\ 
			& A_{1, 2i, 2i}^{(h + 1)} \\ 
		\end{bmatrix} + \begin{bmatrix}
			0 & A_{1,2i-1,2i}^{(h + 1)} \\ 
			A_{1,2i,2i-1}^{(h + 1)} & 0\\ 
		\end{bmatrix}, 
	\]
	and the same holds for $A_{2, jj}^{(h)}$. 
	Then 
	$\delta X^{(h)}_{i,j}$ solves  
	$A_{1,ii}^{(h)} \delta  X^{(h)}_{i,j} + 
	\delta X^{(h)}_{i,j} A_{2,jj}^{(h)} =
	\Xi_{ij}^{(h)}$ where 
	\begin{equation} \label{eq:correction2levels}
	\begin{split}
		\Xi_{ij}^{(h)} := 
		&- \begin{bmatrix}
			0 & A_{1,2i-1,2i}^{(h + 1)} \\ 
			A_{1,2i,2i-1}^{(h + 1)} & 0\\ 
		\end{bmatrix} \begin{bmatrix}
			X_{2i-1, 2i-1}^{(h+1)} & X_{2i-1, 2i}^{(h+1)} \\
			X_{2i, 2i-1}^{(h+1)} & X_{2i, 2i}^{(h+1)} \\
		\end{bmatrix} \\ 
		&- \begin{bmatrix}
			X_{2i-1, 2i-1}^{(h+1)} & X_{2i-1, 2i}^{(h+1)} \\
			X_{2i, 2i-1}^{(h+1)} & X_{2i, 2i}^{(h+1)} \\
		\end{bmatrix} \begin{bmatrix}
			0 & A_{2,2j-1,2j}^{(h + 1)} \\ 
			A_{2,2j,2j-1}^{(h + 1)} & 0\\ 
		\end{bmatrix}.
	\end{split}
	\end{equation}
	Since $X^{(h)}_{ij}$ solves 
	the Sylvester equation 
	\[
		A_{1,ii}^{(h)} X^{(h)}_{ij} + 
		X^{(h)}_{ij} A_{2, jj}^{(h)} = 
		B^{(h)}_{ij}, 
	\]
	then, by rewriting the above equation as a linear system,
	we can bound 
	\[
		\norm{X_{ij}^{(h)}}_F \leq \frac{
			\norm{B_{ij}^{(h)}}_F
		}{\alpha_{1} + \alpha_{2}}. 
	\]
	Applying this relation in \eqref{eq:correction2levels} 
	we get the following bound for the norm of the right-hand side
	$\Xi_{ij}^{(h)}$:
	\begin{align*}
		\norm{\Xi_{ij}^{(h)}}_F &\leq 
		(\beta_{1} + \beta_2) \left\lVert
			\begin{bmatrix}
				X_{2i-1, 2i-1}^{(h+1)} & X_{2i-1, 2i}^{(h+1)} \\
				X_{2i, 2i-1}^{(h+1)} & X_{2i, 2i}^{(h+1)} \\
			\end{bmatrix}
		\right\lVert_F \\ 
		&\leq 
		\frac{\beta_{1} + \beta_2}{\alpha_1 + \alpha_2} 
		\sqrt{
			\norm{B_{2i-1, 2i-1}^{(h+1)}}_F^2 + \norm{B_{2i-1, 2i}^{(h+1)}}_F^2 
			+ \norm{B_{2i, 2i-1}^{(h+1)}}_F^2 + \norm{B_{2i, 2i}^{(h+1)}}_F^2 
		} \\ 
		&= \frac{\beta_{1} + \beta_2}{\alpha_1 + \alpha_2} \norm{B_{ij}^{(h)}}_F.
	\end{align*}
	We define the block matrix $\Xi^{(h)} = [ \Xi_{ij}^{(h)} ]$; collecting all the 
	previous relation as $(i,j)$ varies, we obtain 
	$
	A_{1}^{(h)} \delta X^{(h)} + \delta X^{(h)} A_2^{(h)} = 
	\Xi^{(h)}
	$ and 
	$\norm{\Xi^{(h)}}_F \leq \frac{\beta_{1} + \beta_2}{\alpha_1 + \alpha_2} \norm{B}_F$. 
	
	\subsubsection{Inexact arithmetic}
	
	In a realistic scenario, the Sylvester equations for 
	determining $X^{(\ell)}$ and $\delta X^{(h)}$ are solved
	inexactly. We make the assumption that all Sylvester 
	equations of the form $A_1 X + X A_2 = B$
	are solved with a residual satisfying 
	\begin{equation} \label{eq:sylvester-relative-accuracy}
		A_1 \widetilde{X} + \widetilde{X} A_2 = B + R, \qquad 
		\norm{R}_F \leq \epsilon \norm{B}_F.
	\end{equation}
	Then, the approximate solutions
	computed throughout the recursion verify:
	\begin{align*}
		A_{1, ii}^{(\ell)} \widetilde{X}^{(\ell)}_{i,j} + 
		  \widetilde{X}^{(\ell)}_{i,j} A_{2,jj}^{(\ell)} &= B_{i,j} ^{(\ell)}
		  + R_{ij}^{(\ell)}, \\ 
		  A_{1,ii}^{(h)} \delta \widetilde{X}^{(h)}_{i,j} + 
		  \delta \widetilde{X}^{(h)}_{i,j} A_{2,jj}^{(h)} &=
		  \tilde{\Xi}_{ij}^{(h)} + R_{ij}^{(h)}, 
	\end{align*}
	where $\widetilde{\Xi}_{ij}^{(h)}$ is defined by 
	replacing $X_{ij}^{(h+1)}$ with 
	$\tilde{X}_{ij}^{(h+1)}$ in \eqref{eq:correction2levels}. 
	Thanks to our assumption on the inexact solver, we have that 
	$\norm{R_{ij}^{(\ell)}}_F \leq \epsilon \norm{B_{ij}^{(\ell)}}_F$; 
	bounding $\norm{R_{ij}^{(h)}}_F$ for $h<\ell$, is slightly more challenging, since 
	it depends on the accumulated inexactness. 
	Let us consider the matrices $R^{(h)} = [ R^{(h)}_{ij} ]$
	that correspond to the residuals of the Sylvester equations 
	\begin{align}
		\label{eq:sylv-inexact-0}
		A_{1}^{(\ell)} \widetilde{X}^{(\ell)} + 
		  \widetilde{X}^{(\ell)} A_{2}^{(\ell)} &= B^{(\ell)}
		  + R^{(\ell)}, \\ 
		  \label{eq:sylv-inexact-h}
		  A_{1}^{(h)} \delta \widetilde{X}^{(h)} + 
		  \delta \widetilde{X}^{(h)} A_{2}^{(h)} &=
		  \tilde{\Xi}^{(h)} + R^{(h)}.
	\end{align}
	A bound on $\norm{R^{(h)}}_F$ can be derived
	by controlling the ones of $\tilde{\Xi}^{(h)}$. 
	
	\begin{lemma}
		\label{lem:residuals-rh}
		If the Sylvester equations generated 
		in Algorithm~\ref{alg:dac} are solved with the accuracy 
		prescribed in \eqref{eq:sylvester-relative-accuracy} 
		then $\tilde X^{(h)} := \tilde X^{(\ell)} + 
		\delta \tilde X^{(\ell-1)} + \ldots + \delta \tilde X^{(h)}$ 
		satisfies 
		\begin{equation} \label{eq:tilde-h}
			A_1^{(h)} \tilde X^{(h)} + \tilde X^{(h)}
			A_2^{(h)} = B + R^{(\ell)} + \ldots + R^{(h)}, 
		\end{equation}
		where $R^{(h)}$ are the residuals of \eqref{eq:sylv-inexact-0}
		and \eqref{eq:sylv-inexact-h}.
		In addition, if 
		$\kappa \epsilon < 1$ where $\kappa := \frac{\beta_1 + \beta_2}{\alpha_1 + \alpha_2}$,
		then 
		\[ 
			\norm{R^{(h)}}_F \leq 
			\kappa \epsilon (1 + \epsilon) (1 + \kappa \epsilon)^{\ell-h-1} \norm{B}_F.
		\]
	\end{lemma}
	
	\begin{proof}
		We start with the proof of \eqref{eq:tilde-h} by induction over $h$. 
		For $h = \ell$, the claim follows by \eqref{eq:sylv-inexact-0}. 
		If $h < \ell$, we decompose $\tilde X^{(h)} = \tilde X^{(h+1)} + 
		\delta \tilde X^{(h)}$ to obtain 
		\begin{align*}
			A_1^{(h)} \tilde X^{(h)} + \tilde X^{(h)} A_2^{(h)}
			&= 
			A_1^{(h)} \tilde X^{(h+1)} + \tilde X^{(h+1)} A_2^{(h)}
			+ A_1^{(h)} \delta \tilde X^{(h)} + \delta \tilde X^{(h)} A_2^{(h)} \\ 
			&= \underbrace{ ( A_1^{(h)} - A_1^{(h+1)} ) \tilde X^{(h+1)} + 
			\tilde X^{(h+1)} ( A_2^{(h)} - A_2^{(h+1)} )}_{-\widetilde{\Xi}^{(h)}} \\ 
			&+ A_1^{(h+1)} \tilde X^{(h+1)} + \tilde X^{(h+1)} A_2^{(h+1)}  
			+ \widetilde{\Xi}^{(h)} + R^{(h)}, 
		\end{align*}
		and the claim follows by the induction step. 
	
		We now show the second claim, once again, by induction. For $h = \ell$, we obtain 
		the result by collecting all the residuals together in a block 
		matrix:
		\[
			\norm{R_{ij}^{(\ell)}}_F \leq \epsilon \norm{B_{ij}^{(\ell)}}_F
			\implies 
			\norm{R^{(\ell)}}_F \leq \epsilon \norm{B^{(\ell)}}_F. 
		\]
		Since $\kappa \geq 1$, 
		we have $\epsilon \leq \kappa \epsilon (1 + \epsilon) (1 + \kappa \epsilon)^{-1}$, 
		so that the bound is satisfied. 
		For $h < \ell$ we have 
		\begin{align*}
			\norm{R^{(h)}}_F &\leq \epsilon \norm{\tilde \Xi^{(h)}}_F  
			\leq \epsilon (\beta_1 + \beta_2) \norm{\tilde {X}^{(h+1)}}_F\\
			&\leq \epsilon (\beta_1 + \beta_2) \norm{{X}^{(h+1)}}_F + 
			\epsilon (\beta_1 + \beta_2) \norm{\tilde {X}^{(h+1)} - X^{(h+1)}}_F. 
		\end{align*}
		By subtracting $A_1^{(h+1)} X^{(h+1)} + X^{(h+1)} A_2^{(h+1)} = B$ from 
		\eqref{eq:tilde-h} we obtain 
		\[
			A_1^{(h+1)} (\tilde {X}^{(h+1)} - X^{(h+1)}) + 
			(\tilde {X}^{(h+1)} - X^{(h+1)}) A_2^{(h+1)} = 
			R^{(\ell)} + \ldots + R^{(h+1)}. 
		\]
		Bounding the norm of the solution of this Sylvester equations by 
		$\frac{1}{\alpha_1 + \alpha_2}$ times the norm of the right-hand side yields 
		\[
			\norm{R^{(h)}}_F \leq \kappa \epsilon \left(
				\norm{B}_F + \sum_{j = h+1}^{\ell} \norm{R^{(j)}}_F
			\right). 
		\]
		For $h < \ell$, by the induction step,  we have 
		\begin{align*}
			\norm{R^{(h)}}_F &\leq \kappa \epsilon \left(
				\norm{B}_F + \norm{R^{(\ell)}}_F + \sum_{j = h+1}^{\ell-1} \norm{R^{(j)}}_F
			\right)  \\ 
			&\leq \kappa \epsilon \left(
				1 + \epsilon + 
				  \kappa \epsilon (1 + \epsilon) \sum_{j = h+1}^{\ell-1} 
				  (1 + \kappa \epsilon)^{\ell-j-1}
			\right) \norm{B}_F \\
			&= \kappa \epsilon (1 + \epsilon) \left(1 - \kappa \epsilon \frac{1 - (1 + \kappa \epsilon)^{\ell-h-1}}{\kappa \epsilon}\right) \norm{B}_F \\ 
			&= \kappa \epsilon (1 + \epsilon) (1 + \kappa \epsilon)^{\ell-h-1} \norm{B}_F. 
		\end{align*}
	\end{proof}
	
	We can leverage the previous result to bound the residual of the 
	approximate solution $\tilde X$ returned by 
	Algorithm~\ref{alg:dac2d-balanced}.
	
	\begin{lemma} \label{lem:residual}
		Under the assumptions of Lemma~\ref{lem:residuals-rh}, 
		with the additional constraint $\kappa \epsilon < \frac{2}{\ell}$
		the 
		approximate solution $\tilde{X} := \tilde{X}^{(0)}$ returned by 
		Algorithm~\ref{alg:dac} satisfies 
		\[
			\norm{A_1 \tilde{X} + \tilde{X} A_2 - B}_F 
			\leq (\ell + 1)^2 \kappa \epsilon \norm{B}_F.
		\]
	\end{lemma}
	
	\begin{proof}
		In view of Lemma~\ref{lem:residuals-rh} the residual 
		associated with $\tilde X = \tilde{X}^{(0)}$ satisfies
	 $$
	 \norm{A_1 \widetilde{X}^{(0)} + \widetilde{X}^{(0)} A_2 - B}_F 
	 \leq 
	 \norm{R^{(0)}}_F+ \norm{R^{(1)}}_F+\dots+\norm{R^{(\ell)}}_F. 
	 $$
	 Hence, summing the upper bounds for $ \norm{R^{(h)}}_F$ given 
	 in Lemma~\ref{lem:residuals-rh} we obtain 
	 \begin{align*}
		\norm{A_1 \widetilde{X}^{(0)} + 
		\widetilde{X}^{(0)} A_2 - B}_F 
		&\leq \left( \frac{\kappa \epsilon (1 + \epsilon)}{1 + \kappa \epsilon} 
		  \sum_{h = 0}^{\ell} (1 + \kappa \epsilon)^{\ell-h} \right) \norm{B}_F \\ 
		&\leq 
		  \left[ (1 + \kappa \epsilon)^{\ell + 1} - 1 \right] 
		  \norm{B}_F = 
		\left[ 
			\sum_{h = 1}^{\ell + 1} \binom{\ell + 1}{h} (\kappa \epsilon)^h
		\right] \norm{B}_F. 
	 \end{align*}
	 The assumption $\kappa \epsilon < \frac{2}{\ell}$ guarantees that the 
	 dominant term in the sum occurs for $h = 1$, and therefore we have 
	 \[
		\norm{A_1 \widetilde{X}^{(\ell)} + 
		\widetilde{X}^{(\ell)} A_2 - B}_F 
		\leq
		(\ell + 1)^2 \kappa \epsilon \norm{B}_F. 
	 \]
	\end{proof}
	
	\subsection{Unbalanced dimensions}
	\label{sec:unbalanced-2d}
	
	In the general case, when $n_1 \gg n_2$ and we employ the same 
	$n_{\min}$ for both cluster trees of $A_1$ and $A_2$, we end up 
	with $\ell_1 > \ell_2$. We can artificially obtain $\ell_1 = \ell_2$ 
	by adding $\ell_1 - \ell_2$ auxiliary levels on top of the cluster tree
	of $A_2$. In all these new levels, we consider the trivial 
	partitioning $\{ 1, \ldots, n_2 \} = \{ 1, \ldots, n_2 \} \cup \emptyset$. 
	
	This choice implies that for the first $\ell_1 - \ell_2$ 
	levels of the recursion only the first dimension is split, so that only $2$ 
	matrix equations are generated by these recursive calls. This allows us to 
	extend all the results in Section~\ref{sec:error-analysis-2d} by setting 
	$\ell = \max\{ \ell_1, \ell_2 \}$.
	
	In the practical implementation, for $n_1 \geq n_2$, 
	this approach is encoded in the following steps:
	\begin{itemize}
		\item As far as $n_1$ is significantly larger than $n_2$, e.g. $n_1\geq 2n_2$, 
		apply the splitting on the first mode only. 
		\item When $n_1$ and $n_2$ are almost equal, 
		  apply Algorithm~\ref{alg:dac2d-balanced}. 
	\end{itemize}
	The pseudocode describing this approach is given in Algorithm~\ref{alg:dac2d}.
		
	\begin{algorithm}[h] 
		\small 
		\caption{}\label{alg:dac2d}
		\begin{algorithmic}[1] 
			\Procedure{lyap2d\_d\&c}{$A_1,A_2, B, \epsilon$}
			\If{$\max_i n_i\leq n_{\min}$}
			\State\Return\Call{lyapnd\_diag}{$A_1,A_2, B$}
			\ElsIf{$n_1 \leq 2n_2$ and $n_2 \leq 2n_1$}
				\State \Return \Call{lyap2d\_d\&c\_balanced}{$A_1$, $A_2$, $B$, $\epsilon$}
				\Comment{Algorithm~\ref{alg:dac2d-balanced}}
			\Else 
			\If{$n_1 > 2n_2$}
				\State Partition $B$ as $\left[ \begin{smallmatrix}
					B_1 \\ B_2 
				\end{smallmatrix} \right]$, according to the 
				partitioning in $A_1^{(1)}$
				\State $X_1 \gets$ \Call{lyap2d\_d\&c}{$A_{1,11}^{(1)}, A_2, B_1$},
				 \ $X_2 \gets$ \Call{lyap2d\_d\&c}{$A_{1,22}^{(1)}, A_2, B_2$} 
				\State $X^{(1)} \gets \left[
					\begin{smallmatrix}
						X_1 \\ X_2 
					\end{smallmatrix}
				\right]$
				\State Retrieve a low-rank factorization of 
				  $A_1^{\mathrm{off}} X^{(1)} = UV^T$ \label{lin:mult2d}
			\ElsIf{$n_2 > 2n_1$}
				\State Partition $B$ as $\left[ \begin{smallmatrix}
					B_1 & B_2 
				\end{smallmatrix} \right]$, according to the 
				partitioning in $A_2^{(1)}$
				\State $X_1 \gets$ \Call{lyap2d\_d\&c}{$A_{1}, A_{2,11}^{(1)}, B_1$}, \ 
				$X_2 \gets$ \Call{lyap2d\_d\&c}{$A_1, A_{2,22}^{(1)}, B_2$} 
				\State $X^{(1)} \gets \left[
					\begin{smallmatrix}
						X_1 & X_2 
					\end{smallmatrix}
				\right]$
				\State Retrieve a low-rank factorization of 
				  $X^{(1)} A_2^{\mathrm{off}} = UV^T$
			\EndIf
			\State $\delta X\gets \Call{low\_rank\_sylv}{A_1, A_2, U, V, \epsilon}$
			\State\Return $X^{(1)}+\delta X$
			\EndIf
			\EndProcedure
		\end{algorithmic}
	\end{algorithm}

	\subsection{Solving the update equation}\label{sec:low-rank}
	The update equation \eqref{eq:2dcase} is 
	of the form 
	\begin{equation}\label{eq:lowrank-sylv}
	A_1 \delta X+\delta XA_2 =UV^*
	\end{equation}
	where $A_1, A_2$ are positive definite matrices with spectra
	contained in $[\alpha_1,\beta_1]$ and $[\alpha_2, \beta_2]$, respectively, 
	$U\in\mathbb R^{m\times k}$, $V\in\mathbb R^{n\times k}$ with 
	$k\ll \min\{m,n\}$. Under these assumptions, the singular values
	$\sigma_j(\delta X)$ of the solution $\delta X$ of  
	\eqref{eq:lowrank-sylv} decay rapidly to zero \cite{beckermann2017,penzl2000eigenvalue}. 
	More specifically, it holds
	$$
	\sigma_{1+jk}(\delta X)\leq \sigma_1(\delta X)Z_j([\alpha_1,\beta_2], [-\beta_2, -\alpha_2]),\qquad Z_j(E, F):=\min_{r(z)\in\mathcal R_{j,j}}\frac{\max_E |r(z)|}{\min_F|r(z)|},
	$$
	where $\mathcal R_{j,j}$ is the set of rational functions of the form
	$r(z)=p(z)/q(z)$ having both numerator and denominator of degree at most $j$. 
	The optimization problem associated with $Z_j(E,F)$ is known in the literature as 
	\emph{third Zolotarev problem} and explicit estimates for the decay rate  of $Z_j(E,F)$, as $j$ increases, 
	are available when $E$ and $F$ are disjoint real intervals \cite{beckermann2017}. 
	In particular, 
	we will make use of the following result. 
	\begin{lemma}[\protect{\cite[Corollary 4.2]{beckermann2017}}] \label{lem:zol}
		Let $E= [\alpha_1,\beta_1]\subset \mathbb R^+$, $F=[-\beta_2, -\alpha_2]\subset\mathbb R^-$ be non-empty 
		real intervals, then
		\begin{equation}\label{eq:zol}
		Z_j(E,F)\leq 4\exp\left(\frac{\pi^2}{2\log(16\gamma)}
		\right)^{-2j},\qquad \gamma:=\frac{
			(\alpha_1+\beta_2)(\alpha_2+\beta_1)
		}{
			(\alpha_1+\alpha_2)(\beta_1+\beta_2)
		}.
		\end{equation}
		\end{lemma}
	Lemma~\ref{lem:zol} guarantees the existence of accurate low-rank approximations 
	of $\delta X$. In this setting, the extremal rational function for 
	$Z_j(E,F)$ is explicitly known and close form expressions for its zeros and 
	poles are available.\footnote{Given $E= [\alpha_1,\beta_1]\subset \mathbb R^+$, 
		$F=[-\beta_2, -\alpha_2]\subset\mathbb R^-$
		an expression in terms of elliptic functions 
		for the zeros and poles 
		of the extremal rational function is given  
		in \cite[Eq. (12)]{beckermann2017}. Our implementation is based on 
		the latter.} We will see in the next sections that this enables 
	us to design approximation methods whose convergence rate matches the one
	 in \eqref{eq:zol}.

	\subsubsection{A further property of Zolotarev functions}
	\label{sec:zolotarev-properties}
	We now make an observation that will be relevant for the 
	error analysis in Section~\ref{sec:tensors}. 
	
	Consider the Zolotarev problem associated with
	the symmetric configuration $[\alpha, \beta] \cup [-\beta, -\alpha]$, 
	and let us indicate with $p_1, \ldots, p_s$ and $q_1, \ldots, q_s$ the 
	zeros and poles of the optimal Zolotarev rational function. The symmetry 
	of the configuration yields $p_i = -q_i$, and in turn 
	the bound $|z - p_i| / |z - q_i| \leq 1$ for all 
	$z \in [\alpha, \beta]$. 
	
	The last inequality also holds for nonsymmetric spectra configurations 
	$[\alpha_1, \beta_1] \cup [-\beta_2, -\alpha_2]$. Indeed, the 
	minimax problem is invariant under M\"obius transformations, and 
	the property holds for the evaluations of rational functions on any 
	transformed domains. Since the optimal rational Zolotarev 
	function on $[\alpha_1, \beta_1] \cup [-\beta_2, -\alpha_2]$ can be 
	obtained by remapping the configuration into a symmetric one, 
	the inequality holds on $[\alpha_1,\beta_1]$.

	\subsubsection{Alternating direction implicit method}
	The ADI method iteratively approximates the solution
	of \eqref{eq:lowrank-sylv} with the following two-steps scheme, 
	for given scalar parameters $p_j, q_j \in \mathbb C$:
	\begin{align*}
	(A_1-q_{s+1}I)\delta X_{s+\frac 12}&= UV^*-\delta X_s(A_2+q_{s+1}I),\\
	\delta X_{s+1}(A_2+p_{s+1}I)&=UV^*-(A_1-p_{j+1}I)\delta X_{s+\frac 12}.
	\end{align*}
	The initial guess is $\delta X_0=0$, and it is easy to see that
	$\rank(\delta X_s)\leq sk$. 
	This property is exploited in a specialized version of ADI, which is 
	called \emph{factored ADI} (fADI) \cite{benner2009adi}; the latter 
	computes a factorized form of the update
	$\Delta_s := \delta X_s-\delta X_{s-1} = (q_s - p_s) W_s Y_s^*$ 
	with the following recursion: 
	\begin{align} \label{eq:adi-def}
		\begin{cases}
		  W_{1} = (A_1 - q_1 I)^{-1} U \\ 
		  W_{j+1} = (A_1 - q_{j+1})^{-1} (A_1 - p_{j}) W_j \\ 
		\end{cases} \quad 
		\begin{cases}
			Y_{1} = -(A_2 + p_1 I)^{-1} V \\ 
			Y_{j+1} = (A_2 + p_{j+1})^{-1} (A_2 + q_{j}) Y_j \\ 
		  \end{cases}.
	\end{align}
	The approximate solution $\delta X_s$ after $s$ steps of fADI takes 
	the form 
	\[
		\delta X_s = \Delta_1 + \ldots + \Delta_s = \sum_{j = 1}^s (q_j - p_j) W_j Y_j^*.
	\]
	Observe that, the most expensive operations when executing $s$ steps of fADI are the solution of 
	$s$ shifted linear systems with the matrix $A_1$ and the same 
	amount with the matrix $A_2$. Moreover, the two sequences 
	can be generated independently. 
	
	The choice of the parameters $p_j, q_j$ is crucial to control the 
	convergence of the method, as highlighted by the explicit expressions 
	of the residual and the approximation error after $s$ steps
	\cite{benner2009adi}:
	\begin{align}
		\label{eq:err-adi-full}
		\delta X - \delta X_s &= r_s(A_1) \delta X r_s(-A_2)^{-1}, \\ 
		\label{eq:res-adi-full}
		A_1\delta X_s + \delta X_s A_2 - UV^* &= 
		-r_s(A_1) UV^* r_s(-A_2)^{-1}, 
	\end{align}
	where $r_s(z) = \prod_{j=1}^s\frac{z-p_j}{z-q_j}$
	and the second identity is obtained by applying the operator 
	$X \mapsto A_1 X + XA_2$ to $\delta X_s - \delta X$.
	Taking norms yields the following upper bound for the approximation error:
	\begin{align}
		\label{eq:err-adi}
	\norm{\delta X_s - \delta X}_F &\leq \frac{
		\max_{z\in[\alpha_1,\beta_1]}|r(z)|
	}{
		\min_{z\in[-\beta_2,-\alpha_2]}|r(z)|
	}
	\norm{\delta X}_F,  
	\\
	\label{eq:res-adi}
	\norm{A_1 \delta X_s - \delta X A_2 - UV^*}_F &
	  \leq \frac{\max_{z\in[\alpha_1,\beta_1]}|r(z)|}{\min_{z\in[-\beta_2,-\alpha_2]}|r(z)|}\norm{UV^*}_F, 
	\end{align}
	Inequalities \eqref{eq:err-adi} and 
	\eqref{eq:res-adi} guarantee
	that if the shift parameters $p_j,q_j$ are chosen as the zeros and poles
	of the extremal rational function for 
	$Z_s([\alpha_1,\beta_1], [-\beta_2,-\alpha_2])$, 
	then the approximation error and the residual norm decay 
	(at least) as prescribed by Lemma~\ref{lem:zol}. This allows us to 
	give an a priori bound to the number of steps required to
	achieve a target accuracy $\epsilon$. 
	\begin{lemma} \label{lem:adi-res}
		Let $\epsilon> 0$, $\delta X$ be the solution of \eqref{eq:lowrank-sylv} 
		and $\delta X_s$ the solution returned by applying the 
		fADI method to \eqref{eq:lowrank-sylv} with parameters 
		chosen as the zeros and poles of the extremal rational 
		function for $Z_s([\alpha_1,\beta_1],[-\beta_2,-\alpha_2])$. If 
		\begin{equation}\label{eq:adi-s}
		s\geq \frac{1}{\pi^2}\log\left(\frac{4}{\epsilon}\right)\log\left(
			16\frac{(\alpha_1+\beta_2)(\alpha_2+\beta_1)}{(\alpha_1+\alpha_2)(\beta_1+\beta_2)}
		\right),
		\end{equation}
		then
		\begin{align*}
					 \norm{A_1\delta X_s+\delta X_s A_2-UV^*}_F&\leq \epsilon \norm{UV^*}_F.
		\end{align*}
		\end{lemma} 
	\begin{proof}
		Combining Lemma~\ref{lem:zol} and inequality \eqref{eq:res-adi} yields the claim. 
	\end{proof}
	\subsubsection{Rational Krylov (RK) method}
	Another approach for solving \eqref{eq:lowrank-sylv} is
	 to look for a solution in a well-chosen tensorization 
	 of low-dimensional subspaces. Common choices are rational Krylov
	  subspaces of the form 
	  $\mathcal K_{s,A_1}:=\Span\{U, (A_1-p_1I)^{-1}U, \dots, (A_1-p_sI)^{-1}U\}$, 
	  $\mathcal K_{s,A_2}:=\Span\{V, (A_2^T+q_1I)^{-1}V, \dots, (A_2^T+q_sI)^{-1}V\}$; 
	  more specifically, one consider an approximate solution 
	  $\delta X_s= Q_{s,A_1}\delta YQ_{s,A_2}^*$ where 
	  $Q_{s, A_1},Q_{s,A_2}$ are orthonormal bases of  
	  $\mathcal K_{s,A_1},\mathcal K_{s,A_2}$, 
	  respectively, and $\delta Y $ solves the projected equation
	\begin{equation}\label{eq:proj-sylv}
	(Q_{s,A_1}^*A_1 Q_{s,A_1}) \delta Y +\delta Y(Q_{s,A_2}^* A_2 Q_{s,A_2}) = Q_{s,A_1}^*UV^*Q_{s,A_2}. 
	\end{equation}
	Similarly to fADI, when the parameters $q_j,p_j$ are chosen as the zeros 
	and poles of the extremal rational function for 
	$Z_s([\alpha_1,\beta_1],[-\beta_2,-\alpha_2])$
	then the residual of the approximation can be related to
	\eqref{eq:zol} \cite[Theorem 2.1]{beckermann2011}:
	\begin{equation}\label{eq:residual-sylv}
	\norm{A_1 \delta X_s+\delta X_sA_2-UV^*}_F
	  \leq 2
	\left(1+\frac{\beta_1+\beta_2}{\alpha_1+\alpha_2}\right)
	Z_s([\alpha_1,\beta_1], [-\beta_2, -\alpha_2]) \norm{UV^*}_F.
	\end{equation}
	
	Based on \eqref{eq:residual-sylv} we can state 
	the analogue of Lemma~\ref{lem:adi-res} for 
	rational Krylov. 
	
	\begin{lemma} \label{lem:rk-res}
		Let $\epsilon> 0$, $\delta X$ be the solution of \eqref{eq:lowrank-sylv} 
		and $\delta X_s$ the solution returned by applying the 
		rational Krylov method to \eqref{eq:lowrank-sylv} with shifts 
		chosen as the zeros and poles of the extremal rational function 
		for $Z_s([\alpha_1,\beta_1],[-\beta_2 ,-\alpha_2])$. If 
		\begin{equation}\label{eq:rk-s}
		s\geq \frac{1}{\pi^2}\log\left(\frac{
			 8(\alpha_1+\alpha_2+\beta_1+\beta_2)
		 }{\epsilon(\alpha_1+\alpha_2)}\right)
		 \log\left(16\frac{(\alpha_1+\beta_2)(\alpha_2+\beta_1)}{(\alpha_1+\alpha_2)(\beta_1+\beta_2)}\right),
		\end{equation}
		then
		\begin{align*}
		\norm{A_1 \delta X_s+\delta X_s A_2-UV^*}_F&\leq \epsilon \norm{UV^*}_F.
		\end{align*}
	\end{lemma} 
	
	\subsection{Error analysis for Algorithm~\ref{alg:dac2d}}
	\label{sec:error-analysis-2d}
	
	In Section~\ref{sec:low-rank} we have discussed two possible 
	implementations of \textsc{low\_rank\_sylv} in Algorithm~\ref{alg:dac2d}
	that return an approximate solution of the update equation, with the 
	residual norm controlled by the choice of the parameter $s$. 
	This allows us to use Lemma~\ref{lem:residual} to bound 
	the error in Algorithm~\ref{alg:dac2d}. 
	
	\begin{theorem}\label{thm:2d-accuracy}
		Let $A_1, A_2$ be symmetric positive definite HSS matrices
		with cluster trees of depth at most $\ell$, 
		and spectra 
		contained in $[\alpha_1, \beta_1]$ and $[\alpha_2, \beta_2]$, 
		respectively. 
		Let $\tilde X$ be the approximate solution 
		of \eqref{eq:sylv2d} returned by Algorithm~\ref{alg:dac2d}
		where \textsc{low\_rank\_sylv} is either fADI or RK 
		and uses the $s$ zeros and poles of the extremal rational function
		for $Z_s([\alpha_{1},\beta_{1}],[-\beta_{2}, -\alpha_{2}])$ as input
		parameters.
		If $\epsilon > 0$ and $s$ is
		such that $s \geq s_{\mathrm{ADI}}$ (when 
		\textsc{low\_rank\_sylv} is fADI) or 
		$s \geq s_{\mathrm{RK}}$ (when \textsc{low\_rank\_sylv} is RK)
		where \\ 
		\resizebox{.99\textwidth}{!}{
			\vbox{
		\begin{align*} 
		s_{\mathrm{ADI}} &= 
		\frac{1}{\pi^2}\log\left(4\frac{(\beta_{1}+\beta_{2})}{\epsilon(\alpha_{1}+\alpha_{2})}\right)\log\left(16\frac{(\alpha_{1}+\beta_{2})(\alpha_{2}+\beta_{1})}{(\alpha_{1}+\alpha_{2})(\beta_{1}+\beta_{2})}\right), \\ 
		s_{\mathrm{RK}} &= \frac{1}{\pi^2}\log\left(8\frac{
			(\beta_{1}+ \beta_{2})(\alpha_{1}+\alpha_{2}+\beta_{1}+\beta_{2})
		}{\epsilon(\alpha_{1}+\alpha_{2})^2}\right)\log\left(16\frac{(\alpha_{1}+\beta_{2})(\alpha_{{2}}+\beta_{1})}{(\alpha_{1}+\alpha_{2})(\beta_{1}+\beta_{2})}\right),
		\end{align*}}} \\ 
		and 
		Algorithm~\ref{alg:diag} solves the Sylvester equations 
		at  the base of the recursion with a residual norm bounded 
		by $\epsilon$ times the norm of the right-hand side, then 
		the computed solution satisfies:
		\begin{align}\label{eq:2d-res}
		\norm{A_1 \widetilde{X} + \widetilde{X} A_2- B}_F&\leq
			(\ell + 1)^2 \frac{\beta_1+\beta_2}{\alpha_1+\alpha_2} \epsilon\norm{B}_F. 
		\end{align}
	\end{theorem}
	\begin{proof}
	In view of Lemma~\ref{lem:adi-res} (for fADI)
	and Lemma~\ref{lem:rk-res} (for RK), the assumption 
	on $s$ guarantees that
	the norm of the residual of each update 
	equation is bounded by $\epsilon$ times the norm of its right-hand side. 
	In addition, the equations at the base of the recursion are assumed to be solved 
	at least as accurately as the update equations, and the claim 
	follows by applying Lemma~\ref{lem:residual}. 
	\end{proof}
	
	We remark that, usually, the cluster trees for $A_1$ and $A_2$ 
	are chosen by splitting the indices in a balanced way; this implies 
	that their depths verify $\ell_i \sim \mathcal \mathcal O(\log(n_i / n_{\min}))$. 
	
	\begin{remark}
		We note that $s_{\mathrm{RK}}$ in Theorem~\ref{thm:2d-accuracy}
		is larger than $s_{\mathrm{ADI}}$; under the assumption 
		$\alpha_{1}+\alpha_{2}+\beta_{1}+\beta_{2}\approx \beta_{1}+\beta_{2}$, 
		we have that 
		\[ 
		s_{\mathrm{RK}} - s_{\mathrm{ADI}} \approx \frac{2}{\pi^2} 
		\log \left(2 \frac{\beta_{1} + \beta_{2}}{\alpha_{1} + \alpha_{2}}\right) 
		\log\left(16\frac{(\alpha_{1}+\beta_{2})(\alpha_{2}+\beta_{1})}{(\alpha_{1}+\alpha_{2})(\beta_{1}+\beta_{2})}\right). 
		\]
		In practice, it is often observed that the rational Krylov method requires fewer shifts than fADI 
		to attain a certain accuracy, because its convergence is linked 
		to a rational function which is automatically optimized over a larger 
		space \cite{beckermann2011}. 
	\end{remark}
	
	\subsubsection{The case of M-matrices}
	The term $\kappa$ in the bound of Theorem~\ref{thm:2d-accuracy} 
	arises from controlling the norms of $\Xi^{(h)}$
	and $\tilde \Xi^{(h)} - \Xi^{(h)}$. In the special
	case where $A_j$ are positive definite M-matrices, i.e. $A_j = c_j I-N_j$ with $N_j\geq 0$ and $c_j$ at least as large as the spectral radius of $N_j$, this can be reduced 
	to $\sqrt{\kappa}$ as shown in the following result.
	
	\begin{lemma} \label{lem:m-matrices}
		Let $A_1, A_2$ be symmetric positive definite $M$-matrices. Then, 
		the right-hand sides $\Xi^{(h)}$ of the intermediate Sylvester 
		equations solved in exact arithmetic in Algorithm~\ref{alg:dac2d} 
		satisfy 
		\[
			\norm{\Xi^{(h)}}_F \leq \sqrt{\kappa} \norm{B}_F,  
			\qquad 
			\norm{\tilde \Xi^{(h)} - \Xi^{(h)}}_F \leq \sqrt{\kappa} 
			\norm{R^{(\ell)} + \ldots + R^{(h+1)}}_F,  
		\]
		where $\kappa = (\beta_1 + \beta_2) / (\alpha_1 + \alpha_2)$. 
	\end{lemma}
	
	\begin{proof}
		We note that $\Xi_{ij}^{(h)}$ can be written 
		as $\Xi_{ij}^{(h)} = \mathcal N \mathcal M^{-1} \T \T B_{ij}^{(h)}$, 
		where $\mathcal N$ and $\mathcal M$ are the operators
		\begin{align*}
			\mathcal M Y &:= \begin{bmatrix}
				A_{1, 2i-1, 2i-1}^{(h + 1)} \\ 
				& A_{1, 2i, 2i}^{(h + 1)} \\ 
			\end{bmatrix} Y + Y \begin{bmatrix}
				A_{2, 2i-1, 2i-1}^{(h + 1)} \\ 
				& A_{2, 2i, 2i}^{(h + 1)} \\ 
			\end{bmatrix},  \\ 
		\mathcal N Y &:= -\begin{bmatrix}
			& A_{1, 2i-1, 2i}^{(h + 1)} \\ 
			A_{1, 2i, 2i-1}^{(h + 1)} \\ 
		\end{bmatrix} Y - Y \begin{bmatrix}
			& A_{2, 2i-1, 2i}^{(h + 1)} \\ 
			A_{2, 2i, 2i-1}^{(h + 1)} \\ 
		\end{bmatrix}.
		\end{align*}
		In addition,  $\mathcal M - \mathcal N = 
		I \otimes A^{(h)}_{1,ii} + A^{(h)}_{2,ii} \otimes I$ is a
		positive definite M-matrix, and so is $\mathcal M$. In particular, 
		$\mathcal M - \mathcal N$ is a regular splitting, and therefore 
		$\rho(\mathcal M^{-1} \mathcal N) < 1$. Hence, 
		\begin{align*}
			\norm{\Xi_{ij}^{(h)}}_F &\leq 
			\norm{\mathcal N \mathcal M^{-1}}_2 \norm{\T B_{ij}}_F \\ 
			&\leq 
			\norm{\mathcal M^{\frac 12}}_2
			\norm{\mathcal M^{-\frac 12} \mathcal N \mathcal M^{-\frac 12}}_2
			\norm{\mathcal M^{-\frac 12}}_2
			\norm{\T B_{ij}}_F 
			\leq \sqrt{\kappa} \norm{\T B_{ij}}_F, 
		\end{align*}
		where we have used that the matrix $\mathcal M^{-\frac 12} \mathcal N \mathcal M^{-\frac 12}$ 
		is symmetric and similar to $\mathcal N M^{-1}$,  therefore has both 
		spectral radius and spectral norm bounded by $1$, and that the condition number of $\mathcal M$ is bounded by the one of $I\otimes A_1+A_2\otimes I$. 
	
		For the second relation we have 
		\[
			\tilde \Xi^{(h)} - \Xi^{(h)} = 
			  -(A_1^{(h)} - A_1^{(h+1)}) (\tilde X^{(h+1)} - X^{(h+1)}) - 
			  (\tilde X^{(h+1)} - X^{(h+1)}) (A_2^{(h)} - A_2^{(h+1)}),
		\]
		and $\tilde X^{(h+1)} - X^{(h+1)}$ solves the Sylvester equation 
		(see Lemma~\ref{lem:residuals-rh})
		\[
			A_1^{(h+1)} (\tilde X^{(h+1)} - X^{(h+1)})
			+ (\tilde X^{(h+1)} - X^{(h+1)}) A_2^{(h+1)} = 
			R^{(\ell)} + \ldots + R^{(h+1)}.
		\]
		Following the same argument used for the first point we obtain 
		the claim. 
	\end{proof}
	
	\begin{corollary}\label{cor:accuracy-mmatrices}
		Under the same hypotheses and settings of Theorem~\ref{thm:2d-accuracy}, 
		with the additional assumption of $A_1, A_2$ being M-matrices, 
		Algorithm~\ref{alg:diag} computes an approximate solution 
		$\tilde X$ that satisfies
		\begin{align}\label{eq:2d-res-m}
		\norm{A_1 \widetilde{X} + \widetilde{X} A_2- B}_F&\leq
			(\ell + 1)^2 \sqrt{ \frac{\beta_1+\beta_2}{\alpha_1+\alpha_2} } \epsilon\norm{B}_F. 
		\end{align}
	\end{corollary}
	\subsubsection{Validation of the bounds}
	We now verify the dependency of the final residual norm on the condition number of the problem to inspect the tightness of the behavior predicted by Theorem~\ref{thm:2d-accuracy} and Corollary~\ref{cor:accuracy-mmatrices}. We consider two numerical tests concerning the solution of Lyapunov equations of the form $A_{(i)}X+XA_{(i)}=C$, where the $n\times n$ matrices $A_{(i)}$ have size $n=256$ and increasing condition numbers with respect to $i$; the minimal block size is set to $n_{\min}=32$.
	
	In the first example we generate $A_{(i)}= QD_{(i)}Q^*$ where 
	\begin{itemize}
		\item $D_{(i)}= D^{p_i}$ is the $p_i$th power of the diagonal matrix 
		  $D$ containing the eigenvalues of the 1D discrete Laplacian, i.e. 
		  $2+2\cos(\pi j /(n+1))$, $j=1,\dots, n$. The $p_i$ are chosen with 
		  a uniform sampling of $[1, 2.15]$, so that the condition numbers range 
		  approximately between  $10^4$ and $10^9$.
		\item $Q$ is the Q factor of the QR factorization of a randomly generated matrix with lower bandwidth equal 
		  to $8$, obtained with the MATLAB command \texttt{Q = orth(triu(randn(n), -8), 0)}. This choice 
		  ensures that $A_{(i)}$ is SPD and HSS with HSS rank bounded by $8$~\cite{vandebril2007matrix}.
		\item $C=QSQ^*$, where $Q$ is the matrix defined in the previous point, and $S$ is diagonal with entries 
		  $S_{ii}= \left(\frac{i-1}{n-1}\right)^{10}$. The latter choice aims at giving more importance to the component 
		  of the right-hand side aligned with the smallest eigenvectors of the Sylvester operator. We note that this 
		  is helpful to trigger the worst case behavior of the residual norm.
	\end{itemize} 
	For each value of $i$ we have performed $100$ runs of Algorithm~\ref{alg:dac2d}, using fADI as
	\textsc{low\_rank\_sylv} with threshold $\epsilon=10^{-6}$. The measured residual norms 
	$\norm{A_{(i)}\widetilde X+\widetilde XA_{(i)}-C}_F/\norm{C}_F$ are reported in the left part of 
	Figure~\ref{fig:accuracy}. We remark that the growth of the residual norm appears to be proportional 
	to $\sqrt{\kappa}$ suggesting that the bound from Theorem~\ref{thm:2d-accuracy} is pessimistic.
	
	\begin{figure}
		\centering
		\begin{tikzpicture}
			\begin{loglogaxis}[legend pos = north west, xlabel = $\kappa$, width=.5\linewidth, ymax=1e-1, ymin=2e-8, title={General SPD}]
				\addplot[only marks, mark=x, red] table {test_accuracy.dat};
				\addplot[domain = 1e4:4e9, dashed, blue]{1e-7 * sqrt(x)};
				\legend{Res.\ norm, $\mathcal O(\sqrt{\kappa})$}
			\end{loglogaxis}
		\end{tikzpicture}~\begin{tikzpicture}
		\begin{loglogaxis}[legend pos = north west, xlabel = $\kappa$, width=.5\linewidth, ymax=1e-1, ymin=2e-8, title={SPD M-matrix}]
			\addplot[only marks, mark=x, red] table {test_accuracy_mmatrix.dat};			
			\legend{Res.\ norm}
		\end{loglogaxis}
	\end{tikzpicture}
	\caption{Residual norm behavior for Algorithm~\ref{alg:dac2d}, with respect to the conditioning of problem. On the left, the coefficients of the matrix equation are generic HSS SPD matrices; On the right, the coefficients have the additional property of being M-matrices.}\label{fig:accuracy}
	\end{figure}
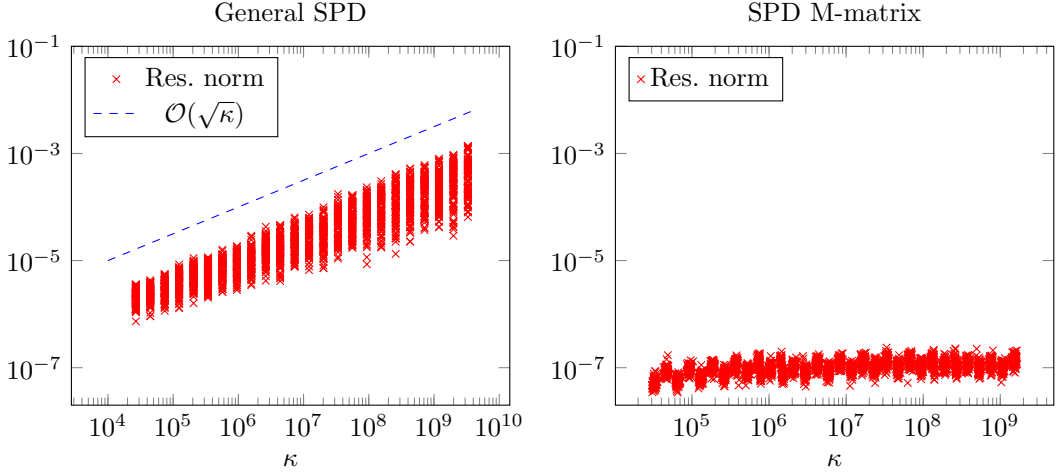
	
	In the second test, the matrices $A_{(i)}$ are chosen as shifted 1D discrete Laplacian, where the shift is selected to match the same range of condition number of the first experiment; note that, the matrices $A_{(i)}$ are SPD M-matrices, with HSS rank 1 (they are tridiagonal). The right-hand side $C$ is chosen as before, by replacing $Q$ with the eigenvector matrix of the 1D discrete Laplacian. Corollary~\ref{cor:accuracy-mmatrices} would predict an $\mathcal O(\sqrt{\kappa})$ growth for the residual norms; however, as shown in the right part of Figure~\ref{fig:accuracy}, the residual norms seems not influenced by the condition number of the problem.
	\begin{remark}
	The examples reported in this section have been chosen to display the worst case behavior of the residual norms; in the typical case, for instance by choosing \texttt{C = randn(n)}, the influence of the condition number on the residual norms is hardly visible also in the case of non M-matrix coefficients.	
	\end{remark}
	\subsection{Complexity of Algorithm~\ref{alg:dac2d}}
	
	 In order to simplify the complexity analysis of Algorithm~\ref{alg:dac2d}
	 we assume that $n_1=2^{\ell_1} n_{\min}$ 
	 and $n_2=2^{\ell_2} n_{\min}$ and 
	 that $A_1,A_2$ are HSS matrices of HSS rank $k$, with a partitioning 
	 of depth $\ell_1, \ell_2$ obtained by halving the dimension at every level. 
	
	We start by considering the simpler case $n := n_1 = n_2$ and therefore 
		$\ell = \ell_1 = \ell_2$. 
	   We remark that given $s\in\mathbb N$, executing $s$ steps 
	   of fADI or RK for solving \eqref{eq:reshape-update}
	   with size $n \times n$, requires
	   \begin{align}\label{eq:sylv-cost}
	   \mathcal C_{ADI}(n, s, k) &= \mathcal O(k^2sn),&
	   \mathcal C_{RK}(n, s, k) &=\mathcal O(k^2s^2n),
	   \end{align}
	   flops, respectively. Indeed, each iteration of fADI involves the solution 
	   of two block linear systems with $2k$ columns each and an HSS matrix 
	   coefficient; this is performed by computing two ULV factorizations
	   ($\mathcal O(k^2n)$, see \cite{xia2010fast}) and solving linear 
	   systems with triangular and unitary HSS matrices ($\mathcal O(kn)$ 
	   for each of the $2k$ columns). The cost analysis of rational Krylov 
	   is analogous, but it is dominated by reorthogonalizing 
	   the basis at each iteration; this requires $\mathcal O(k^2jn)$ flops at 
	   iteration $j$, for $j=1,\dots,s$.  
	   
	   Combining these ingredients yields the following.
	   \begin{theorem}\label{thm:2d-complexity}
	   Let $A_1,A_2\in\mathbb R^{n\times n}$, $n=2^\ell n_{\min}$, and assume
	   that $A_1,A_2$ are HSS matrices of HSS rank $k$, with a partitioning
	   of depth $\ell$ obtained by halving the dimension at every level. 
	   Then, Algorithm~\ref{alg:dac2d-balanced} requires:
	   \begin{itemize}
		   \item[$(i)$] $\mathcal O((k\log(n)+n_{\min}+ sk^2)n^2)$ flops 
			if \textsc{low\_rank\_sylv} implements $s$ steps of the fADI method,
		   \item[$(ii)$] $\mathcal O((k\log(n)+n_{\min}+ s^2k^2)n^2)$ flops if \textsc{low\_rank\_sylv} implements $s$ steps of the rational Krylov method.
	   \end{itemize}
	   \end{theorem}  
	\begin{proof}
	We only prove $(i)$	since $(ii)$ is obtained with an analogous argument. 
	Let us analyze the cost of Algorithm~\ref{alg:dac} at level $j$ of the
	recursion. For $j=\ell$ it solves $2^\ell \cdot 2^\ell$ Sylvester equations 
	of size $n_{\min}\times n_{\min}$ by means of Algorithm~\ref{alg:diag}; 
	this requires $\mathcal O(4^\ell\cdot n_{\min}^3)=\mathcal O(n_{\min}n^2)$ 
	flops. At level $j<\ell$, $4^j$ Sylvester equations of size 
	$\frac{n}{2^j}\times \frac{n}{2^j}$ and right-hand side of rank 
	(at most) $2k$ are solved via fADI. According to \eqref{eq:sylv-cost}, 
	the overall cost of these is $\mathcal O(2^jk^2sn)$. Finally, we need
	 to evaluate $4^j$ multiplications at line~\ref{lin:mult2d} which yields
	  a cost $\mathcal O(4^j k(n/2^j)^2)= \mathcal O(kn^2)$. Summing over 
	  all levels $j=0,\dots,\ell-1$ we get 
	  $\mathcal O(\sum_{j=0}^{\ell-1} (kn^2+2^jk^2sn))= \mathcal O(\ell kn^2 +2^{\ell}k^2s n)$. Noting that $\ell=\mathcal O(\log(n))$ provides the claim. 
	\end{proof}
	
	We now consider the cost of Algorithm~\ref{alg:dac2d} in the more general case $n_1 \neq n_2$. Without loss of generality, we assume
	$n_1=2^{\ell_1}n_{\min}> 2^{\ell_2}n_{\min}=n_2$; 
	Algorithm~\ref{alg:dac2d} begins with $\ell_1-\ell_2$ splittings
	 on the first mode. This generates $2^{\ell_1-\ell_2}$ subproblems 
	 of size $n_2\times n_2$, $2^{\ell_1-\ell_2}-1$ calls to 
	 \textsc{low\_rank\_sylv} and $2^{\ell_1-\ell_2}-1$ multiplications 
	 with a block vector at line~\ref{lin:mult}. In particular, we have 
	 $2^j$ calls  to \textsc{low\_rank\_sylv} for problems of size 
	 $n_1/2^j\times n_2$, $j=0,\dots, \ell_1-\ell_2-1$; this yields
	 an overall cost of $\mathcal O((\ell_1-\ell_2)sk^2n_1)$ 
	 and $\mathcal O((\ell_1-\ell_2)s^2k^2n_1)$ when fADI and rational Krylov 
	 are employed, respectively. Analogously, the multiplications
	 at line~\ref{lin:mult2d} of this initial phase require 
	 $\mathcal O((\ell_1-\ell_2)kn_1n_2)$ flops.
	Summing these costs to the estimate provided by Theorem~\ref{thm:2d-complexity}, multiplied by $2^{\ell_1-\ell_2}$, yields the following corollary.
	\begin{corollary}\label{cor:2d-complexity-diff}
		Under the assumptions of Theorem~\ref{thm:2d-complexity}, apart
		from $n_1=2^{\ell_1}n_{\min}\geq n_2=2^{\ell_2}n_{\min}$ with 
		$n_2\geq \log(n_1/n_2)$, we have that Algorithm~\ref{alg:dac2d} costs:
		   \begin{itemize}
			\item[$(i)$] $\mathcal O((k\log(n_1)+n_{\min}+ sk^2)n_1n_2)$ if \textsc{low\_rank\_sylv} implements $s$ steps of the fADI method,
			\item[$(ii)$] $\mathcal O((k\log(n_1)+n_{\min}+ s^2k^2)n_1n_2)$ if \textsc{low\_rank\_sylv} implements $s$ steps of the rational Krylov method.
		\end{itemize}
		\end{corollary} 
	
	\section{Tensor Sylvester equations} \label{sec:tensors}
	We now proceed to describe the procedure sketched in the 
	introduction for $d > 2$. Initially, we assume that all
	dimensions
	are equal, 
	so that the splitting is applied to all modes at every
	step of the recursion.  
	
	We identify two major differences with respect to the 
	case $d = 2$, both concerning the update equation \eqref{eq:update-all}:
	\begin{equation*} \label{eq:update-all2}
		\delta \T X\times_1 A_1+\delta\T X\times_2 A_2+\dots+\delta\T X\times_d A_d 
		=- \sum_{t = 1}^d  \T X^{(1)} \times_t A_t^{\mathrm{off}}.
	\end{equation*}
	First, the latter equation cannot be 
	solved directly with a single call to 
	\textsc{low\_rank\_sylv}, since the right-hand side
	does not have a low-rank matricization. However, 
	the $t$-mode matricization of 
	the term $\T X^{(1)} \times_t A_t^{\mathrm{off}}$ 
	has rank bounded by $k$ for $t = 1, \ldots, d$. 
	Therefore,
	the update equation can be addressed by 
	$d$ separate calls to the low-rank solver,
	by splitting the right-hand side, and matricizing the $t$th term 
	in a way that separates the mode $t$ from the rest, as follows:
	\begin{equation} \label{eq:matr-t}
		A_t \delta X_t + 
		\delta X_t \sum_{j \neq t}
		  I \otimes \cdots \otimes I \otimes A_j \otimes I \otimes \cdots \otimes I 
		= - A_t^{\mathrm{off}} X^{(1)}, 
	\end{equation}
	where, in this case, $X^{(1)}$ denotes 
	the $t$-mode matricization of $\T X^{(1)}$. 
	
	Second, the solution of \eqref{eq:matr-t}
	by means of \textsc{low\_rank\_sylv} requires 
	solving shifted linear systems with a Kronecker sum of 
	$d - 1$ matrices, that is performed by recursively 
	calling the divide-and-conquer scheme. This generates 
	an additional layer of inexactness, that we have to take 
	into account in the error analysis. This makes the analysis 
	non-trivial, and requires results on the error propagation
	in \textsc{low\_rank\_sylv}; in the next section we address 
	this point in detail for fADI, that will be the method 
	of choice in our implementation. 
	
	In the case where the dimensions $n_i$ are unbalanced, 
	we follow a similar approach to the case $d = 2$. More precisely, 
	at each level 
	we only split on the $r$ dominant modes which
	satisfy $2n_i \geq \max_j  n_j $, and which are 
	larger than $n_{\min}$. This generates $2^r$ recursive 
	calls and $r$ update equations.
	
	We summarize the procedure in Algorithm~\ref{alg:dac}. 
\begin{remark}
Other recursive solvers for tensor and matrix equations (such as \textsc{recsy}~\cite{chen2020recursive}) adopt a different strategy and always split one mode at a time. Here we pursue the simultaneous splitting approach because it comes with a few computational advantages. In particular, the low-rank Sylvester solvers for the different modes can be run in
parallel and the depth of the recursion tree is reduced. We only resort on splitting one (or some) mode at a time when the dimensions are unbalanced because it facilitates the analysis of the computational complexity. 
\end{remark}
	
	\begin{algorithm}[H] 
		\small 
		\caption{}\label{alg:dac}
		\begin{algorithmic}[1]
			\Procedure{lyapnd\_d\&c}{$A_1,A_2,\dots, A_d,\T B, \epsilon$}
			\If{$\max_i n_i\leq n_{\min}$}
			\State\Return\Call{lyapnd\_diag}{$A_1,A_2,\dots, A_d,\T B$}
			\Else
			\State Permute the modes to have 
			  $2n_i \geq \max_j \{ n_j \}$
			  and $n_i \geq n_{\min}$, for $i \leq r$. 
			\For{$j_1, \ldots, j_r = 1,2$}
				\State $\T Y^{(j_1, \ldots, j_r)} \gets $\Call{lyapnd\_d\&c}{$A_1^{(j_1 j_1)},\dots, A_s^{(j_r j_r)},A_{r+1},\dots, A_d,\T B^{(j_1, \ldots, j_r)}, \epsilon$}
				\State Copy $\T Y^{(j_1, \ldots, j_r)}$ in the corresponding 
				  entries of $\T X^{(1)}$. 
			\EndFor
			\For{$j = 1, \ldots, r$}
				\State Retrieve a rank $k$ factorization $A_j^{\mathrm{off}}=U\widetilde V^T$
				\State $V\gets \Call{reshape}{-\T X^{(1)} \times_j \widetilde V, \prod_{i\neq j}n_i, k}$ \label{lin:mult}
				\State $\delta X_j \gets \Call{low\_rank\_sylv}{A_j, \sum_{i\neq j} I\otimes\dots \otimes I\otimes A_i\otimes I\otimes\dots\otimes I , U, V, \epsilon}$ \label{lin:lr-sylv}
				\State $\delta \T X_j \gets $ tensorize $\delta X_j$, inverting the $j$-mode matricization. 
			\EndFor
			\State \Return $\T X^{(1)} + \delta\T X_1 + \ldots + \delta \T X_r$.
			\EndIf
			\EndProcedure
		\end{algorithmic}
	\end{algorithm}
	
	\subsection{Error analysis}
	The use of an inexact solver for tensor Sylvester equations with 
	$d - 1$ modes in \textsc{low\_rank\_sylv} has an impact on the
	achievable accuracy of the method. Under the assumption that all the 
	update equations at line \ref{lin:lr-sylv} are solved with a relative 
	accuracy $\epsilon$ we can easily generalize the analysis
	performed in Section~\ref{sec:d=2-exact-equations}. 
	
	More specifically, we consider the 
	additive decomposition of $\tilde{\T X} = \tilde{\T X}^{(0)} = 
	\tilde{\T X}^{(\ell)} + \delta \tilde{\T X}^{(\ell-1)} + \ldots + 
	\delta \tilde{\T X}^{(0)}$, 
	where the equations solved by the algorithm are the following:
	\begin{align}
		\label{eq:sylvd-inexact-0}
		\sum_{t = 1}^d \widetilde{\T X}^{(\ell)} \times_{t} A_t^{(\ell)} &= \T B
		  + \T R^{(\ell)} \\ 
		\label{eq:sylvd-inexact-h}
		  \sum_{t = 1}^d \delta \widetilde{\T X}^{(h,j)} \times_t A_t^{(h)}
		  &= \tilde{\Xi}^{(h,j)} + \T R^{(h,j)}, 
	\end{align}
	with $A_t^{(h)}$ are the block-diagonal matrices defined 
	in \eqref{eq:hodlr-level-h}, 
	and $\delta \tilde {\T X}^{(h)} = 
	\delta \tilde {\T X}^{(h,1)} + \ldots + \delta \tilde {\T X}^{(h,d)}$, 
	and $\tilde \Xi^{(h,j)}$ is defined as 
	follows:
	\[
		\tilde \Xi^{(h,j)} := 
		- \tilde{\T X}^{(h -1)} \times_j (A_j^{(h)} - A_j^{(h+1)}). 
	\]
	We remark that 
	the matrix $A_j^{(h)} - A_j^{(h+1)}$ contains the off-diagonal blocks 
	that are present in $A_j^{(h)}$ but not in $A_j^{(h+1)}$. 
	When the dimensions $n_i$ are unbalanced we artificially 
	introduce some levels in the HSS partitioning (see Section~\ref{sec:unbalanced-2d}), 
	some of these may be the zero matrix. 
	Then, we state the higher-dimensional version of 
	Lemma~\ref{lem:residual}. 
	
	\begin{lemma}	\label{lem:residual-d>2}
		If the tensor Sylvester equations \eqref{eq:sylvd-inexact-0} 
		and \eqref{eq:sylvd-inexact-h} are solved 
		with the relative accuracies
		$\norm{\T{R^{(\ell)}}}_F \leq \epsilon \norm{\T B}_F$ and 
		$\norm{\T{R}^{(h,j)}}_F \leq \epsilon \norm{\tilde \Xi^{(h,j)}}_F$ and 
		$\kappa \epsilon < 1$,
		with $\kappa := \frac{\beta_1 + \ldots + \beta_d}{\alpha_1 + \ldots + \alpha_d}$,
		then the approximate solution $\tilde{\T X}$ returned by 
		Algorithm~\ref{alg:dac}
		satisfies
		\[
			\Big\lVert \sum_{i = 1}^d \tilde{\T X} 
			\times_i A_i - \T B \Big\rVert_F \leq 
			(\ell + 1)^2 \kappa \epsilon \norm{\T B}_F.
		\]
	\end{lemma}
	
	\begin{proof}
		Let us consider $\tilde{\T X}^{(h)} = \tilde{\T X}^{(h)}
		+ \delta \tilde{\T X}^{(\ell-1)} + \ldots + \delta \tilde{\T X}^{(h)}$. 
		Following the same argument in the proof of Lemma~\ref{lem:residuals-rh}, 
		we obtain 
		\[
			\sum_{i = 1}^d \tilde{\T X}^{(h)} 
			\times_i A_i = \T B + 
			\T R^{(\ell)} + \ldots + \T R^{(h)}, 
		\]
		where $\T R^{(h)} := \sum_{j = 1}^d \T R^{(h,j)}$. Hence, we 
		bound 
		\begin{align*}
			\norm{\T R^{(h)}}_F &\leq \sum_{j = 1}^d \norm{\T R^{(h,j)}}_F
			\leq \epsilon \sum_{j = 1}^d \norm{\tilde \Xi^{(h,j)}}_F\\ 
			&\leq \epsilon \sum_{j = 1}^d \beta_j (\norm{\T X^{(h+1)}}_F
			+ \norm{\T X^{(h+1)} - \tilde {\T X}^{(h+1)}}_F) \\
			&\leq \epsilon \kappa (
				\norm{\T B}_F + \norm{\T R^{(\ell)}}_F + \ldots + \norm{\T R^{(h+1)}}_F
			). 
		\end{align*}
		By induction one can show that 
		\[
			\norm{\T R^{(h)}}_F \leq \kappa \epsilon (1 + \epsilon) 
			  (1 + \kappa \epsilon)^{\ell-h-1} \norm{\T B}_F,
		\]
		and the claim follows applying the same reasoning as in the proof 
		of Lemma~\ref{lem:residual}. 
	\end{proof}
	
	Lemma~\ref{lem:residual-d>2} ensures that if the update equations are 
	solved with uniform accuracy the growth in the residual is controlled 
	by a moderate factor, as in the matrix case. 
	
	We now investigate what can be ensured if we perform a constant number 
	of fADI steps throughout all levels of recursions (including the use of 
	the nested Sylvester solvers). This requires the development of technical
	tools for the analysis of factored ADI with inexact solves.

	\subsubsection{Factored ADI with inexact solves}
	Algorithm~\ref{alg:dac} solves update equations 
	that can be matricized as $A_1 \delta X + \delta X A_2 = UV^*$, where the factors $U$ and $V$ are 
	retrieved analogously to the matrix case, see Remark~\ref{rem:rhs}, and
	$A_1$ is the Kronecker sum of $d - 1$ matrices. In particular, 
	linear systems with $A_1$ are solved inexactly by a nested call to 
	Algorithm~\ref{alg:dac}. 
	In this section we investigate how the inexactness affects the 
	residual of the solution computed with $s$ steps of fADI. 
	
	We begin by recalling some properties of fADI. 
	Assume to have carried out $j+1$ steps 
	of fADI applied to the equation $A_1 \delta X + \delta X A_2 = UV^*$. 
	Then, the factors 
	$W_{j+1}, Y_{j+1}$ can be obtained by running a single 
	fADI iteration for the equation 
	\[
		A_1 \delta X + \delta X A_2 = -r_j(A_1) UV^* r_j(-A_2)^{-1}, 
		\quad 
		r_j(z):=\prod_{h=1}^j\frac{z-p_h}{z-q_h}, 
	\]
	using parameters $p_{j+1}, q_{j+1}$. 
	
	We now consider the sequence $\tilde W_j$ obtained by solving 
	the linear systems with $A_1$ inexactly:
	\begin{equation} \label{eq:Wj}
		(A_1 - q_{j+1}I) \widetilde W_{j+1} = 
			  (A_1 - p_{j} I) \widetilde W_j + \eta_{j+1}, 
			\quad 
			 (A_1 - q_1I) \widetilde W_1 = U + \eta_1,
	\end{equation}
	where $\norm{\eta_{j+1}}_2 \leq \epsilon \norm{U}$. 
	
	The following lemma quantifies the distance between 
	$\tilde W_j$ and $W_j$. Note that we make 
	the assumption $|z - p_j| \leq |z - q_j|$ over 
	$[\alpha_{1}, \beta_{1}]$, that is satisfied by the
	parameters of the Zolotarev functions, 
	as discussed in Section~\ref{sec:zolotarev-properties}. 
	
	\begin{lemma} \label{lem:tildeWjsmall}
		Let $A_1$ be positive definite, and 
		$p_j, q_j$ satisfying 
		$|z - p_j| \leq |z - q_j|$ for any
		$z \in [\alpha_{1}, \beta_{1}]$. 
		Let $\widetilde W_j$ be the sequence defined as in 
		\eqref{eq:Wj}
		Then, it holds 
		\[
			(A - p_{j} I) \widetilde W_j = r_j(A_1) U + M_j, \qquad 
			\norm{M_j}_F \leq j\epsilon \norm{U}_F, \qquad 
			r_j(z) := \prod_{i = 1}^j \frac{z - p_i}{z - q_i}.
		\]
	\end{lemma}
	
	\begin{proof}
		For $j = 1$, the claim follows directly from the assumptions. 
		For $j > 1$, we have 
		\begin{align*}
			(A_1 - p_{j+1}I) \widetilde W_{j+1} &= 
			(A_1 - p_{j+1}I) (A_1 - q_{j+1}I)^{-1}
			\left[ 
			  (A_1 - p_{j} I) \widetilde W_j + 
			  \eta_{j+1} \right] \\ 
		&= \frac{r_{j+1}}{r_{j}}(A_1)
		\left[ 
		  r_j(A_1) U + M_j + 
		  \eta_{j+1} \right] \\
		&= r_{j+1}(A_1) U + \frac{r_{j+1}}{r_{j}}(A_1) \left( M_{j} + \eta_{j+1} \right). 
		\end{align*}
		The claim follows by setting 
		$M_{j+1} := 
		\frac{r_{j+1}}{r_{j}}(A_1) \left(M_{j} + \eta_{j+1}\right)$, 
		and using the property 
		$r_{j+1}(z) / r_{j}(z) = |z - p_{j+1}| / |z - q_{j+1}| \leq 1$ 
		over the spectrum of $A_1$. 
	\end{proof}
	
	\begin{remark}
		We remark that the level of accuracy required in \eqref{eq:Wj} is guaranteed up to second 
		order terms in $\epsilon$, whenever the residual norm for the linear systems 
		$(A_1 - q_{j+1}I) \widetilde W_{j+1}= (A_1 - p_{j} I) \widetilde W_j + \eta_{j+1}$ 
		satisfies  $\norm{\eta_{j+1}}\leq \epsilon\norm{(A_1 - p_{j} I) \widetilde W_j}_F$. 
		Indeed, the argument used in the proof of Lemma~\ref{lem:tildeWjsmall} can be easily modified to get 
		$$\norm{M_j}_F\leq [(1+\epsilon)^j-1]\norm{U}_F=j\epsilon\norm{U}_F+\mathcal O(\epsilon^2).$$
		The slightly more restrictive choice made in \eqref{eq:Wj}, allows us to obtain more readable results.
	\end{remark}

	We can then quantify the impact of carrying out
	the fADI iteration with inexactness in one of the two sequences on the residual norm. 
	
	\begin{theorem}\label{thm:adi-res-inex}
		Let us consider the solution of \eqref{eq:lowrank-sylv}
		by means of the fADI method using shifts 
		satisfying the property $|z - p_j| \leq |z - q_j|$ 
		for any	$z \in [\alpha_{1}, \beta_{1}]$, 
		and let $\epsilon >0$ such that 
		the linear systems defining $\widetilde W_j$ are 
		solved inexactly as in \eqref{eq:Wj}. 
		Then, the computed solution 
		$\delta \widetilde X_s$ verifies:
		\[
			\norm{A_1 \delta \widetilde X_s + \delta \widetilde X_s A_2 - UV^*}_F \leq 
			  \epsilon_{s,\mathrm{ADI}}
			  + 
			  2 s \epsilon\norm{U}_F \norm{V}_2, 
		\]
		where 
		$\epsilon_{s,\mathrm{ADI}} := \norm{A_1 \delta X_s + \delta X_s A_2 - UV^*}_F$
		is the norm of the residual 
		after $s$ steps 
		of the fADI method in exact arithmetic.
		\end{theorem}
	\begin{proof}
		We indicate with $\delta \widetilde X_j$
		the inexact solution computed at step $j$ of fADI. 	
		We note that 
		$\delta \widetilde X_1$ corresponds to the 
		outcome of one exact 
		fADI iteration for the slightly modified right-hand side 
		$(U + \eta_1)V^*$; hence, by  \eqref{eq:res-adi-full}
		$\delta \widetilde X_1$ satisfies 
		the residual equation:
		\begin{align*}
			A_1 \delta \widetilde X_1 + \delta \widetilde X_1 A_2 - (U + \eta_1) V^* 
			&= 
			  - r_1(A_1) (U + \eta_1) V^* r_1(-A_2)^{-1},
		\end{align*}
		which allows to express the residual on the original equation 
		as follows:
		\begin{align*} 
			A_1 \delta \widetilde X_1 + \delta \widetilde X_1 A_2 - U V^*  
			&= -r_1(A_1) (U + \eta_1) V^* r_1(-A_2)^{-1} + \eta_1 V^*. 
		\end{align*}
		We now derive an analogous result for the 
		update $\Delta_{j+1} := \delta \widetilde X_{j+1} - \delta \widetilde X_j = 
		(q_{j+1} - p_{j+1}) \widetilde W_{j+1} Y_{j+1}^*$
		where $\Delta_1 =  \delta \widetilde  X_1$ by setting $\delta \widetilde X_0 = 0$. 
		We prove that for any $j$ the correction $\Delta_{j+1}$ satisfies 
		\begin{align*}
			A_1 \Delta_{j+1} + 
			\Delta_{j+1} A_2 &= 
			\left[
				(A_1 - p_jI) \widetilde{W}_j + \eta_{j+1}
			\right] V^* r_{j}(-A_2)^{-1} \\
			& -
			\frac{r_{j+1}}{r_j}(A_1) \left[
				(A_1 - p_{j} I) \widetilde{W}_{j} + \eta_{j+1}
			\right] V^* r_{j+1}(-A_2)^{-1}, \qquad 
			j \geq 1. 
		\end{align*}
		We verify the claim for $j = 1$; $\widetilde{W}_2$ is defined 
		by the relation 
		\[
			(A_1 - q_2 I) \widetilde{W}_2 = 
			  (A_1 - p_1 I) \widetilde W_1 + \eta_2 
			  = r_1(A_1) (U + \eta_1) + \eta_2.
		\]
		Hence, as discussed in the beginning of this proof, 
		$\widetilde{W}_2$ is the outcome of one 
		exact iteration of 
		fADI applied to the equation 
		$A_1 \delta X + \delta X A_2 - (r_1(A_1) (U + \eta_1) + \eta_2 )V^* r_1(-A_2)^{-1} = 0$. 
		Hence, thanks to \eqref{eq:res-adi-full} 
		the residual of the computed update 
		$\Delta_2$ verifies 
		\begin{align*}
			A_1 \Delta_2 + 
			\Delta_2 A_2 &- (r_1(A_1) (U + \eta_1) + \eta_2 )V^*r_1(-A_2)^{-1}\\
			&= -\frac{r_2}{r_1}(A_1) 
			  \left[
				r_1(A_1) (U + \eta_1) + \eta_2 )V^*
			  \right]
			  r_2(-A_2)^{-1} 
		\end{align*}
	
		Then, 
		the claim follows for $j = 1$ noting that  
		$(A - p_1 I) \widetilde W_1 = r_1(A_1) (U + \eta_1)$ 
		using the first relation in 
		\eqref{eq:Wj}.
		For $j > 1$ we have 
		$(A_1 - q_{j+1} I) \widetilde{W}_{j+1} = 
		(A - p_j I) \widetilde W_j + \eta_{j+1}$
		and with a similar argument $\Delta_{j+1}$ 
		is obtained as a single fADI iteration in exact arithmetic
		for the equation 
		$A_1 \delta X + \delta X A_2 - ((A - p_j I) \widetilde W_j + \eta_{j+1}) V^* r_j(-A_2)^{-1} = 0$, 
		and in view of the residual equation~\eqref{eq:res-adi-full}
		\begin{align*}
			A_1 \Delta_{j+1} + \Delta_{j+1} A_2 &- 
			  \left[
				(A_1 - p_j I) \widetilde W_j + \eta_{j+1}
			  \right] V^* r_j(-A_2)^{-1} \\
			  &= -
			  \frac{r_{j+1}}{r_j}(A_1) \left[
				(A_1 - p_j I) \widetilde W_j + \eta_{j+1}
			  \right] V^* r_{j+1}(-A_2)^{-1}. 
		\end{align*}
		We now write $\delta \widetilde X_s = \sum_{j = 1}^{s} \Delta_j$, 
		so that by linearity 
		the residual associated with $\delta \widetilde X_s$ satisfies 
		\begin{align*}
			A_1 \delta \widetilde X_s + \delta \widetilde X_s A_2 - UV^* &= 
			A_1 \Delta_1 + \Delta_1 A_2 - UV^* + 
			\sum_{j = 1}^{s-1} (A_1 \Delta_{j+1} + \Delta_{j+1} A_2) \\ 
			&= -r_1(A_1) (U + \eta_1) V^* r_1(-A_2)^{-1} + \eta_1 V^* \\ 
			&+ \sum_{j = 1}^{s-1} \Bigg\{ -\frac{r_{j+1}}{r_j}(A_1) \left[
				(A_1 - p_j I) \widetilde W_j + \eta_{j+1}
			  \right] V^* r_{j+1}(-A_2)^{-1} \\
			  &+ \left[
				(A_1 - p_j I) \widetilde W_j + \eta_{j+1}
			  \right] V^* r_j(-A_2)^{-1} \Bigg\}. 
		\end{align*}
		We now observe that, thanks to \eqref{eq:Wj} 
		\[
			\frac{r_{j+1}}{r_j}(A_1) 
			\left[  
				(A_1 - p_j I) \widetilde W_j - 
				\eta_{j+1} \right] = 
				(A_1 - p_{j+1} I) \widetilde W_{j+1}.
		\]
		Plugging this identity in the summation yields 
		\begin{align*} 
			& \sum_{j = 1}^{s-1} \resizebox{.95\linewidth}{!}{$\Bigg\{
				\left[
				(A_1 - p_j I) \widetilde W_j + \eta_{j+1}
			  \right] V^* r_j(-A_2)^{-1}  -\frac{r_{j+1}}{r_j}(A_1) \left[
				(A_1 - p_j I) \widetilde W_j + \eta_{j+1}
			  \right] V^* r_{j+1}(-A_2)^{-1}  
			  \Bigg\}$} \\ 
			&= \sum_{j = 1}^{s-1} \Bigg\{\left[
				(A_1 - p_j I) \widetilde W_j + \eta_{j+1}
			  \right] V^* r_j(-A_2)^{-1} -(A - p_{j+1} I) \widetilde W_{j+1} V^* r_{j+1}(-A_2)^{-1}  
			  \Bigg\} \\ 
			&=(A - p_1 I) \widetilde W_1 V^* r_1(-A_2)^{-1} -(A_1 - p_sI) \widetilde W_s V^* r_s(-A_2)^{-1}
			 + \sum_{j=1}^{s-1} \eta_{j+1} V^* r_{j}(-A_2)^{-1}
		\end{align*}
		Summing this with the residual yields 
		\[
			A_1 \delta \widetilde X_s + \delta \widetilde X_s A_2 - UV^* = 
			-(A - p_sI) \widetilde W_s V^* r_s(-A_2)^{-1} +
			\sum_{j = 0}^{s-1} \eta_{j+1} V^* r_{j}(-A_2)^{-1}, 
		\]
		where we have used the relation 
		$r_1(A_1)(U + \eta_1) = (A - p_1 I) \widetilde W_1$
		from \eqref{eq:Wj}. In view Lemma~\ref{lem:tildeWjsmall}
		we can write $(A - p_sI) \widetilde W_s = r_s(A_1) U + M_s$, 
		where $\norm{M_s}_F \leq s\epsilon$. Therefore, 
		\begin{align*}
			\norm{ A_1 \delta \widetilde X_s + \delta \widetilde X_s A_2 + UV^* }_F &\leq 
			\norm{r_s(A_1) UV^* r_s(-A_2)^{-1}}_F + s\epsilon \norm{U}_F \norm{V}_2 \\
			&+ 
			  \sum_{j = 0}^{s-1} \norm{\eta_{j+1} V^* r_{j}(-A_2)^{-1}}_F \\ 
			&\leq 
			\epsilon_{ADI, s} + 2 s \epsilon \norm{U}_F \norm{V}_2. 
		\end{align*}
		\end{proof}
	
		\begin{remark} \label{rem:normUV}
			We note that, the error in Theorem~\ref{thm:adi-res-inex} depends on 
			$\norm{U}_F \norm{V}_2$; this may be larger than $\norm{UV^*}_F$, 
			which is what we need to ensure the relative accuracy of the 
			algorithm. However, under the additional assumption 
			that $V$ has orthogonal columns, we have 
			$\norm{UV^*}_F = \norm{U}_F \norm{V}_2$. We can always 
			ensure that this condition is satisfied by computing a thin 
			QR factorization of $V$, and right-multiplying $U$ by $R^*$. This 
			does not increase the complexity of the algorithm.
		\end{remark}
	
	\subsubsection{Residual bounds with inexactness}
	
	We can now exploit the results on inexact fADI to control 
	the residual norm of the approximate solution returned by 
	Algorithm~\ref{alg:dac}, assuming that all the update equations 
	are solved with a fixed number $s$ of fADI steps with optimal shifts. 
	
	\begin{theorem} \label{thm:accuracy-d}
	Let $A_1,\dots, A_d\in\mathbb R^{n\times n}$, symmetric positive
	definite with spectrum contained in $[\alpha, \beta]$, and 
	$\kappa := \frac{\beta}{\alpha}$. Moreover, assume 
	that the $A_j$s are HSS matrices of HSS rank $k$, with a partitioning of 
	depth $\ell$.	 
	Let $\epsilon>0$ and suppose that \textsc{low\_rank\_sylv} uses fADI with
	the $s$ zeros and poles of the extremal rational function for
	$Z_s([\alpha,\beta],[-(d-1)\beta, -\alpha])$ as input parameters, 
	with right-hand side reorthogonalized as described in Remark~\ref{rem:normUV}.
	If
	$$
	s\geq \frac{1}{\pi^2}\log\left(2\frac{d\kappa}{\epsilon}\right)\log\left(8\frac{(\alpha+(d-1)\beta)(\alpha+\beta)}{d\alpha\beta}\right)
	$$ and Algorithm~\ref{alg:diag} solves the Sylvester equations 
	at  the base of the recursion with residual bounded 
	by $\epsilon$ times the norm of the right-hand side,
	then the solutions $\widetilde{\mathcal X}$ computed by 
	Algorithm~\ref{alg:dac} satisfies:
	\begin{align*}
	\norm{\widetilde{\mathcal X}\times_1 A_1+\dots+ \widetilde{\mathcal X}\times_d A_d-\T B}_F\leq
	  \left(
		  \left[ 
			 \kappa (\ell + 1)^{2} \right]^{d-1} (1 + 2s)^{d-2}\epsilon\right) \norm{\T B}_F.
	\end{align*}
	\end{theorem}
	\begin{proof}
		Let $\epsilon_{\mathrm{lr}, d}$ be the relative residual at which the
		low-rank update equations with $d$ modes 
		are solved in the recursion and let $\epsilon_d:= (\ell+1)^2 \kappa \epsilon_{\mathrm{lr}, d}$.
		Note that, thanks to Lemma~\ref{lem:residual-d>2}, we have
		\[
			\norm{\widetilde{\mathcal X}\times_1 A_1+\dots+ 
			\widetilde{\mathcal X}\times_d A_d-\T B}_F 
			\leq 
			\epsilon_d \norm{\T B}_F , 
		\]
		so that $\epsilon_d$ is an upper bound for the relative
		residual of the target equation. Moreover, using the error bound for inexact fADI 
		of Theorem~\ref{thm:adi-res-inex}, we can write $\epsilon_{\mathrm{lr}, d} \leq (1 + 2s) \epsilon_{d-1}$, 
		which implies 
		\[
			\epsilon_d \leq \begin{cases}
				(\ell+1)^2 \kappa \epsilon 
				& d = 2 \ \ (\text{Theorem~\ref{thm:2d-accuracy}})
				\\
				(\ell+1)^2 \kappa (1 + 2s) \epsilon_{d-1}, &
				d \geq 3. 	
			\end{cases}, 
		\]
		where $\epsilon$ is $Z_s([ \alpha, \beta ], [-(d-1)\beta, -\alpha])$. 
		Expanding the recursion  yields the sought bound. 
		\end{proof}
	
		Theorem~\ref{thm:accuracy-d} bounds the residual error with a constant 
		depending on $\kappa^{d-1}$, which can often be pessimistic. This term 
		arises when bounding $\norm{\Xi^{(h,j)}}_F$ with $\norm{A_j^{(h)}}_2$ 
		multiplied by $\norm{\T X^{(h+1)}}_F$. When the $A_j$s are M-matrices, 
		this can be improved, by replacing $\kappa$ with $\sqrt{\kappa}$. 
	
		\begin{corollary}
			Under the same hypotheses of Theorem~\ref{thm:accuracy-d} and the 
			additional assumption that the $A_t$s are symmetric positive definite $M$-matrices, we have 
			\[
				\norm{\widetilde{\mathcal X}\times_1 A_1+\dots+ \widetilde{\mathcal X}\times_d A_d-\T B}_F\leq
				\left(
					\left[ 
						d \sqrt{\kappa} (\ell + 1)^{2} \right]^{d-1} (1 + 2s)^{d-2}\epsilon\right) \norm{\T B}_F.
			\]
		\end{corollary}
		\begin{proof}
			By means of the same argument used in Lemma~\ref{lem:m-matrices}
			we have 
			that $\norm{\Xi^{(h,j)}} \leq \sqrt{\kappa} \norm{\T B}_F$
			and $\norm{\Xi^{(h,j)} - \tilde \Xi^{(h,j)}}_F \leq \sqrt{\kappa} 
			\norm{\T R^{(\ell)} + \ldots + \T R^{(h+1)}}_F$. Plugging these 
			bounds in the proof of Lemma~\ref{lem:residual-d>2} yields 
			the inequality 
			\begin{align*}
				\norm{\T R^{(h)}}_F &\leq 
				  \epsilon \sum_{j=1}^d 
				  \left[ 
					  \norm{\Xi^{(h,j)}}_F + \norm{\tilde \Xi^{(h,j)} - \Xi^{(h,j)}}_F
				  \right] \\
				  &\leq 
				  d \epsilon \sqrt{\kappa} \left[
					  \norm{\T B}_F + \norm{\T R^{(\ell)}}_F + \ldots + \norm{\T R^{(h+1)}}_F
				  \right].
			\end{align*}
			Then, following the same steps of Lemma~\ref{lem:residual-d>2} yields 
			\[
				\norm{\T R^{(h)}}_F \leq 
				  d \sqrt{\kappa} \epsilon (1 + \epsilon) (1 + d\sqrt{\kappa} \epsilon)^{\ell-h-1} 
				  \norm{\T B}_F.
			\]
			Using this bound in Theorem~\ref{thm:accuracy-d} yields the claim. 
		\end{proof}

	\subsection{Complexity analysis}
	Theorem~\ref{thm:2d-complexity} can be generalized to the $d$-dimensional 
	case, providing a complexity analysis when nested solves are used. 
	
	\begin{theorem} \label{thm:3d-complexity}
		Let $A_i  \in\mathbb R^{n_i\times n_i}$  
		$n_i = 2^{\ell_i} n_{\min}$, with $n_1\geq n_2\geq \dots \geq n_d$, $n_d\geq skd$, $n_d\geq \log(n_1/n_d)$, and assume that $A_i$ are 
		HSS matrices of HSS rank $k$, 
		with a partitioning of depth $\ell_i$ obtained by halving
		the dimension at every level, for $i=1,\dots,d$. Then, Algorithm~\ref{alg:dac} costs:
		\begin{itemize}
			\item[$(i)$] $\mathcal O((k(d + \log(n_1))+n_{\min}+ sk^2)n_1\dots n_d)$ if \textsc{low\_rank\_sylv} implements $s$ steps of the fADI method,
			\item[$(ii)$] $\mathcal O((k(d + \log(n_1))+n_{\min}+ s^2k^2)n_1\dots n_d)$ if \textsc{low\_rank\_sylv} implements $s$ steps of the rational Krylov method.
		\end{itemize}
	\end{theorem}
	\begin{proof}
		We only prove $(i)$ because $(ii)$ is completely analogous.
		
		 Let us assume that we have $r$ different mode sizes and each of those occurs $d_h$ times, i.e.:
		\begin{align*}
		\underbrace{n_1=n_2=\dots=n_{i_1}}_{d_1}> \underbrace{n_{i_1+1}=\dots=n_{i_2}}_{d_2}>\dots >\underbrace{n_{i_{r-1}+1}=\dots=n_{i_r}}_{d_r}.
		\end{align*} We proceed by (bivariate) induction over 
		$d\geq 2$ and $\ell_1\geq 0$; the cases $d=2$ and $\ell_1\geq 0$  
		are given by Theorem~\ref{thm:2d-complexity} and 
		Corollary~\ref{cor:2d-complexity-diff}. For $d>2$ and $\ell_1=0$, 
		we have that $n:=n_1=\dots=n_d=n_{\min}$ and 
		Algorithm~\ref{alg:dac} is equivalent to Algorithm~\ref{alg:diag} 
		whose cost is $\mathcal O(n^{d+1})= \mathcal O(n_{\min}n^d)$; so also 
		in this case the claim is true. For $d>2$ and $\ell_1>0$, the algorithm 
		begins by  halving the first $d_1$ modes generating $2^{d_1}$ 
		subproblems 
		having dominant size $n_1/2=2^{\ell_1-1}n_{\min}$. By induction these 
		subproblems cost
		\begin{align*}
		\mathcal O&\left(2^{d_1}(k(d + \log(n_{1}/2))+n_{\min}+sk^2)(n_{i_1}/2)^{d_1}n_{i_2}^{d_2}n_{i_3}^{d_3}\dots n_{i_r}^{d_r}\right)\\ 
		&=O\left((k(d + \log(n_{1}))+n_{\min}+sk^2)\prod_{i=1}^dn_i\right).
		\end{align*}
		Then, we focus on the cost of the update equations and of the tensor times (block) vector multiplications,
		 in this first phase of Algorithm~\ref{alg:dac}.
		The procedure generates $d_1$ update equations of size $n_1\times\dots\times n_d$.
		The cost of each call to \textsc{low\_rank\_sylv} is dominated by the complexity of solving $sk$ (shifted) 
		linear systems with a Kronecker sum structured matrix with $d-1$ modes. By induction, 
		the cost of all the calls to  \textsc{low\_rank\_sylv}  is bounded by
		\begin{align*}
	&\mathcal \mathcal O\left( skd_1(k(d + \log(n_1))+n_{\min}+sk^2)\prod_{j=1}^{d-1}n_j\right)\\
	&=\mathcal O\left( (k(d+\log(n_1))+n_{\min}+sk^2)\prod_{j=1}^{d}n_j\right).
	\end{align*}	
	Finally, since the tensor times (block) vector multiplications are in one-to-one correspondence 
	with the calls to \textsc{low\_rank\_sylv}, we have that the algorithm generates $d_1$ products of 
	complexity $\mathcal O(kn_1\dots n_d)$. Adding the contribution $\mathcal O(dkn_1\dots n_d)$ to the 
	cost of the subproblems provides the claim.
	\end{proof}	
		
	\section{Numerical experiments} \label{sec:numerical-experiments}
	We now test the proposed algorithm against some implementations of Algorithm~\ref{alg:diag} 
	where  the explicit diagonalization of the matrix coefficients 
	$A_t$ is either done in dense arithmetic or via the algorithm proposed in \cite{ou2022superdc}. 
	Note that,  the dense solver delivers accuracy close to machine precision while the other approaches 
	aim at a trading of some accuracy for a speedup in 
	the computational time. 
	We assess this behavior on 2D and 3D examples. 
	
	\subsection{Details of the implementation}
	An efficient implementation of Algorithm~\ref{alg:dac2d} and 
	Algorithm~\ref{alg:dac} takes some care. In particular: 
	
	\begin{itemize}
		\item In contrast to the numerical tests of Section~\ref{sec:error-analysis-2d}, 
		the number of Zolotarev shifts for fADI and RK is adaptively chosen on each level of 
		the recursion to ensure the accuracy described in \eqref{eq:sylvester-relative-accuracy}. 
		More precisely, this requires estimates of the spectra of the matrix coefficients $A_t^{(h)}$ at all 
		levels of recursion $h$; this is done via the MATLAB built-in function \texttt{eigs}.
		  However, since $A^{(h)}_i$ appears in $2^{(d-1)h}$ equations, 
		  estimating the spectra in each recursive call would incur in redundant 
		  computations. Therefore, we precompute estimates of the spectra for each 
		  block before starting Algorithm~\ref{alg:dac2d} and \ref{alg:dac}, by walking 
		  the cluster tree. 
		\item Since the $A_t$s are SPD, we remark that the 
		  correction equation can be slightly modified to obtain a right-hand side 
		  with half of the rank. This is obtained by replacing $A_j^{\mathrm{off}}$
		  with a low-rank matrix having suitable non-zero diagonal blocks. See 
		  \cite[Section~4.4.2]{kressner2019low} for the details on this idea. This 
		  is crucial for problems with higher off-diagonal ranks, and is used in the 
		  2D Fractional cases described below. 
		\item A few operations in Algorithm~\ref{alg:dac} are well suited for parallelism: 
		  the solution of the Sylvester equations at the base of the recursions
		  are all independent, and the same holds for the 
		  body of the for loop at lines 11--14. We exploit this fact by computing all 
		  the solutions in parallel using multiple cores in a shared memory 
		  environment. 
		  \item When the matrices $A_t$ are both HSS and banded, they are represented within the sparse format and the sparse direct solver of MATLAB (the backslash operator) is used for the corresponding system solving operations. Note that, in this case the peculiar location of the nonzero entries makes easy to construct the low-rank factorizations of the off-diagonal blocks.
	\end{itemize}
	In addition to fADI and rational Krylov, we consider another popular
	 low-rank solver for Sylvester equations: the \emph{extended Krylov} 
	 method~\cite{simoncini2007new} (EK). The latter corresponds to the rational 
	 Krylov method where the shift parameters alternate between the values 
	 $0$ and $\infty$. 
	 In particular, EK's iteration leverages the precomputation of either the Cholesky 
	 factorization of the sparse coefficient matrices, or the ULV factorization~\cite{xia2010fast} in the HSS case. This is convenient because the shift parameters do not vary. 
	 A slight downside is that we do not have a priori bounds on the error, and 
	 we have to monitor the residual norm throughout the iterations to detect
	 convergence.
	
	An implementation of the proposed algorithms is freely available at 
	\url{https://github.com/numpi/teq_solver}, 
	and requires \texttt{rktoolbox}\footnote{\url{https://rktoolbox.org}} 
	\cite{berljafa2015generalized}
	and \texttt{hm-toolbox}\footnote{\url{https://github.com/numpi/hm-toolbox}} 
	\cite{massei2020hm} as external 
	dependencies. The repository contains the numerical experiments 
	included in this document, and includes a copy of the SuperDC 
	solver.\footnote{\url{https://github.com/fastsolvers/SuperDC}} \cite{ou2022superdc}
	
	The experiments have
	been run on a server with two Intel(R) Xeon(R) E5-2650v4 CPU with 12 cores and 24 threads
	each, running at 2.20 GHz, using MATLAB R2021a with the Intel(R) Math Kernel Library
	Version 11.3.1. The examples have been run using the SLURM scheduler, 
	allocating $8$ cores and $240$ GB of RAM.
	
	\subsection{Laplace and fractional Laplace equations}
	\label{sec:2d-laplace-test}
	In this first experiment we validate the asymptotic complexity of 
	Algorithm~\ref{alg:dac2d} and compare the performances of various 
	low-rank solvers for the update equation. As case studies we select 
	two instances of the matrix equation $AX + XA= B$. In one 
	case  $A\in\mathbb R^{n\times n}$ is chosen as the usual central 
	finite difference discretization of  the 1D Laplacian. In the other case 
	$A$ is the Gr\"unwald-Letnikov finite difference discretization of the 1D
	 fractional Laplacian with order $1.5$, see  \cite[Section 2.2.1]{massei2019fast}. In both cases the
	 reference solution $X$ is randomly generated by means of the MATLAB command 
	 \texttt{randn(n)}, and $B$ is computed as $B := AX + XA$. 
	 We remark that for both equations the matrix $A$ is 
	 SPD and HSS; in particular for the Laplace equation $A$ is tridiagonal and 
	 stored in the sparse format, while for the fractional Laplace equation it 
	 has the rank structure depicted in Figure~\ref{fig:hss-fract}. For the Laplace 
	 equation we set $n_{\min}=512$. In view of the higher off-diagonal ranks of the 
	 fractional case, we consider the larger minimal block size $n_{\min}=2048$. 
	 \begin{figure}
		 \centering
		 \includegraphics[width=.3\textwidth, height=.3\textwidth]{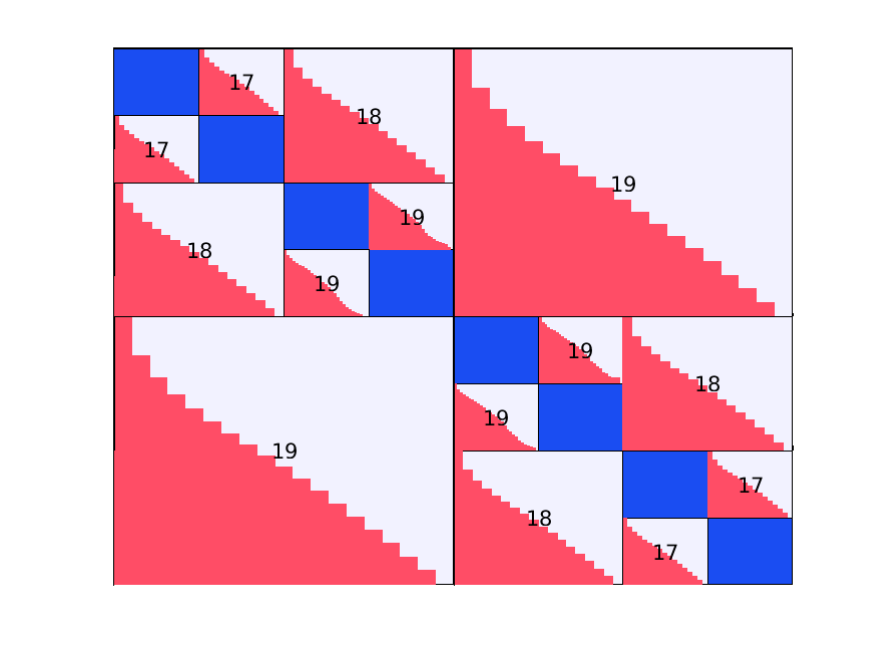}
		 \caption{Hierarchical low-rank structure of the 1D fractional 
		 Laplace operator discretized through Gr\"unwald-Letnikov finite 
		 differences. The blue blocks are dense, while the gray blocks 
		 are stored as low-rank matrices whose rank is indicated by the 
		 number in the center. In red, the magnitude of the singular values of the low-rank blocks.}
		 \label{fig:hss-fract}
	 \end{figure}
	 In the next section we will investigate  how varying this parameter affects the performances. 
	
	We consider increasing sizes $n=2^j$, $j=10,\dots, 15$ and the following solvers:
	\begin{description}
		\item[diag] Algorithm~\ref{alg:diag} 
			with explicit diagonalization of the matrix $A$ performed in dense arithmetic.
		 \item[dst] Algorithm~\ref{alg:diag} incorporating the fast diagonalization 
			 by means of the Discrete Sine Transform (DST). This approach is only considered 
			 for the 2D Laplace equation. 
		\item[superdc] solver implementing Algorithm~\ref{alg:diag} using the fast diagonalization 
		  for SPD HSS matrices described in \cite{ou2022superdc}. 
		\item[dc\_adi] Algorithm~\ref{alg:dac2d} where the fADI iteration is used as \textsc{low\_rank\_sylv}. 
		\item[dc\_rk] Algorithm~\ref{alg:dac2d} where the rational Krylov method is used as \textsc{low\_rank\_sylv}.
		\item[dc\_ek] Algorithm~\ref{alg:dac2d} where the extended Krylov method is used as \textsc{low\_rank\_sylv}.
	\end{description} 
	The shifts used in \textbf{dc\_adi} and \textbf{dc\_rk} are the optimal 
	Zolotarev zeros and poles. The number of shifts is 
	chosen to obtain similar residual norms
	 $\mathrm{Res}:=\norm{A\widetilde X + \widetilde X A -B}_F/\norm{B}_F$ 
	 of about $10^{-10}$. 
	
	We start by comparing the different implementation of Algorithm~\ref{alg:dac2d}
	with \textbf{diag}. A detailed comparison with \textbf{dst} and \textbf{superdc}
	is postponed to Section~\ref{sec:superdc}. 
	
	The running times and residuals are reported in Table~\ref{tab:2d-laplace}
	for the Laplace equation. The fractional case is reported 
	in Table~\ref{tab:2d-fractional}, for which we do not report the timings for $n=1024,2048$ 
	since our choice of $n_{\min}$ makes Algorithm~\ref{alg:dac2d} equivalent to Algorithm~\ref{alg:diag}.
	
	For both experiments, using fADI as low-rank solver yields the cheapest method. We remark that, 
	in the fractional case, \textbf{dc\_ek} outperforms \textbf{dc\_rk} since the precomputation of the 
	Cholesky factorization of $A$ (and of its sub blocks) makes the iteration of extended Krylov significantly 
	cheaper than the one of rational Krylov. 
	
	In Figure~\ref{fig:hist_2D_Laplace} and \ref{fig:hist_2D_Fractional_Laplace} 
	we display how the time is distributed among the various subtasks of \textbf{dc\_adi}, i.e., 
	the time spent on solving dense matrix equations, computing the low-rank updates, 
	forming the RHS of the update equation and updating the solution, and
	estimating the spectra.

	\begin{table}
		\caption{Timings and residuals for the solution of the 2D 
		Laplace equation of size $n \times n$ by means of Algorithm~\ref{alg:dac2d}
		using different low rank solvers, and $n_{\min} = 512$.}
		\resizebox{\textwidth}{!}{ 
		\begin{filecontents*}{test1.dat}
%% LaTeX2e file `test1.dat'
%% generated by the `filecontents' environment
%% from source `paper' on 2023/01/17.
%%
  512 0.099554 1.2005e-13 0.18573 1.2005e-13 0.055668 1.2005e-13 0.050743 1.2005e-13
 1024 0.2207 2.9074e-13 0.57544 2.4926e-10 0.5524 1.4229e-10 0.54659 3.0559e-10
 2048 0.94992 3.8668e-13 1.0375 3.1349e-10 1.9686 1.4744e-10 2.4169 3.987e-10
 4096 5.8621 1.2454e-12 3.5119 3.4297e-10 8.4206 1.4802e-10 10.535 4.3992e-10
 8192 41.883 3.3708e-12 14.924 3.6532e-10 35.467 1.4917e-10 43.186 4.6517e-10
 16384 320.76 7.8448e-12 61.409 3.7013e-10 149.63 1.4913e-10 181.48 4.8035e-10
 32768 2511.3 8.8846e-12 258.06 3.7227e-10 625.83 1.4902e-10 759.95 4.9323e-10
 
\end{filecontents*}
\pgfplotstabletypeset[skip rows between index={0}{1},
		every head row/.style={
			before row={
				\toprule
				\multicolumn{1}{c|}{}&
				\multicolumn{2}{c|}{\textbf{diag}}&
	\multicolumn{2}{c|}{\textbf{dc\_adi}}&
	\multicolumn{2}{c|}{\textbf{dc\_rk}}&
	\multicolumn{2}{c}{\textbf{dc\_ek}}
				\\
				\multicolumn{1}{c|}{}&
				\multicolumn{2}{c|}{}&
				\multicolumn{2}{c|}{$n_{\min} = 512$}&
				\multicolumn{2}{c|}{$n_{\min} = 512$}&
				\multicolumn{2}{c}{$n_{\min} = 512$}
				\\
			},
			after row = \midrule,
		},
		columns/0/.style = {column name = $n$, column type=c|},
			columns/1/.style = {column name = Time,precision=1,zerofill, fixed},
		columns/2/.style = {column name = Res,precision=1,zerofill, column type=c|},
			columns/3/.style = {column name = Time,precision=1,zerofill, fixed},
	columns/4/.style = {column name = Res,precision=1,zerofill, column type=c|},
			columns/5/.style = {column name = Time,precision=1,zerofill, fixed},
	columns/6/.style = {column name = Res,precision=1,zerofill, column type=c|},
			columns/7/.style = {column name = Time,precision=1,zerofill, fixed},
	columns/8/.style = {column name = Res,precision=1,zerofill, column type=c},
		]{test1.dat}}
		\label{tab:2d-laplace}
	\end{table}
	
	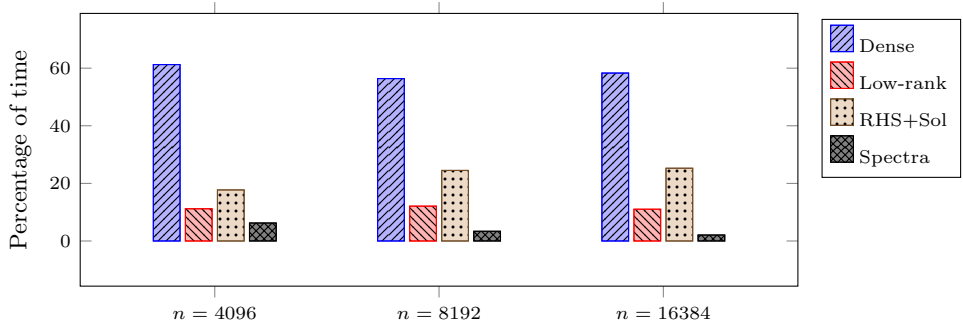
\begin{figure}
		\centering
		\begin{tikzpicture}
			\begin{axis}[
			title = \text{2D Laplace, $n_{\min} = 512$},
			x tick label style={
			/pgf/number format/1000 sep=},
		ylabel=\small{Percentage of time},
		width = .75\linewidth, height = .25\textheight,
		symbolic x coords={$n=4096$, $n=8192$, $n=16384$},
		xtick = data,
		tick label style={font=\scriptsize},
		enlargelimits=0.3,
		ybar,
		legend style={font=\scriptsize,
			cells={anchor=west}, anchor = north west, xshift = 0.5cm
		},
	legend image code/.code={
		\draw [#1] (0.2cm,-0.05cm) rectangle (0.5cm,0.3cm); },
	],
		legend pos = outer north east,
		]
			\addplot+[postaction={
				pattern=north east lines
			}] coordinates { ($n=4096$,61.25) ($n=8192$,56.38) ($n=16384$,58.30)  };
			\addplot+[postaction={
				pattern=north west lines
			}] coordinates { ($n=4096$,11.19) ($n=8192$,12.10) ($n=16384$,11.03)  };
			\addplot+[postaction={
				pattern=dots
			}] coordinates { ($n=4096$,17.73) ($n=8192$,24.50) ($n=16384$,25.26)  };
			\addplot+[postaction={
				pattern=crosshatch
			}] coordinates { ($n=4096$,6.25) ($n=8192$,3.38) ($n=16384$,2.08)  };
			
			\legend{Dense,Low-rank,RHS+Sol,Spectra}
			\end{axis}
			\end{tikzpicture}	
			\caption{Distribution of the time spent in the different tasks in 
			  Algorithm~\ref{alg:dac2d}. The results are for some instances 
			  of the 2D Laplace equation considered in
			  Section~\ref{sec:2d-laplace-test} with fADI as low-rank solver
			  and $n_{\min} = 512$.}
			\label{fig:hist_2D_Laplace}
	\end{figure}

	\begin{table}
		\caption{Timings and residuals for the solution of the 2D 
		fractional Laplace equation of size $n \times n$ by means of Algorithm~\ref{alg:dac2d}
		using different low rank solvers, and $n_{\min} = 2048$.}
		\resizebox{\textwidth}{!}{ 
			\begin{filecontents*}{test1_bis.dat}
%% LaTeX2e file `test1_bis.dat'
%% generated by the `filecontents' environment
%% from source `paper' on 2023/01/17.
%%
  512 0.088435 4.1337e-15 0.22348 4.1337e-15 0.083585 4.1337e-15 0.08796 4.1337e-15
 1024 0.35876 5.1055e-15 0.36213 5.1055e-15 0.37264 5.1055e-15 0.31666 5.1055e-15
 2048 2.0805 6.2894e-15 2.3068 6.2894e-15 2.0386 6.2894e-15 2.397 6.2894e-15
 4096 15.974 8.2916e-15 20.927 4.5479e-11 24.723 1.6174e-11 19.112 1.2235e-10
 8192 137.15 1.7874e-14 99.444 5.8513e-11 120.69 2.4795e-10 97.786 1.2255e-10
 16384 1061.1 2.303e-14 439.02 6.8474e-11 593.23 2.1602e-10 471.29 1.5583e-10
 32768 8297 3.1923e-14 1998 3.0864e-10 2948.1 3.557e-10 2467.1 2.9557e-10
 
\end{filecontents*}
\pgfplotstabletypeset[skip rows between index={0}{3},
			every head row/.style={
				before row={
					\toprule
					\multicolumn{1}{c|}{}&
					\multicolumn{2}{c|}{\textbf{diag}}&
					\multicolumn{2}{c|}{\textbf{dc\_adi}}&
					\multicolumn{2}{c|}{\textbf{dc\_rk}}&
					\multicolumn{2}{c}{\textbf{dc\_ek}}
					\\
					\multicolumn{1}{c|}{}&
					\multicolumn{2}{c|}{}&
					\multicolumn{2}{c|}{$n_{\min} = 2048$}&
					\multicolumn{2}{c|}{$n_{\min} = 2048$}&
					\multicolumn{2}{c}{$n_{\min} = 2048$}
					\\
				},
				after row = \midrule,
			},
			columns/0/.style = {column name = $n$, column type=c|},
			columns/1/.style = {column name = Time,precision=1,zerofill, fixed},
			columns/2/.style = {column name = Res,precision=1,zerofill, column type=c|},
			columns/3/.style = {column name = Time,precision=1,zerofill, fixed},
			columns/4/.style = {column name = Res,precision=1,zerofill, column type=c|},
			columns/5/.style = {column name = Time,precision=1,zerofill, fixed},
			columns/6/.style = {column name = Res,precision=1,zerofill, column type=c|},
			columns/7/.style = {column name = Time,precision=1,zerofill, fixed},
			columns/8/.style = {column name = Res,precision=1,zerofill, column type=c},
			]{test1_bis.dat}}
		\label{tab:2d-fractional}
	\end{table}
	
	\begin{figure}
		\centering
		\begin{tikzpicture}
			\begin{axis}[
			title = \text{Fractional 2D Laplace, $n_{\min} = 2048$},
			x tick label style={
				/pgf/number format/1000 sep=},
			ylabel=\small{Percentage of time},
			width = .75\linewidth, height = .25\textheight,
			symbolic x coords={$n=4096$, $n=8192$, $n=16384$},
			xtick = data,
			tick label style={font=\scriptsize},
			enlargelimits=0.3,
			legend style={anchor=north west},
			ybar,
			legend style={
				font=\scriptsize,
				cells={anchor=west}, xshift = 0.5cm
			},
		legend image code/.code={
			\draw [#1] (0.2cm,-0.05cm) rectangle (0.5cm,0.3cm); },
		],
			legend pos = outer north east,
			]
			\addplot+[postaction={
				pattern=north east lines
			}] coordinates { ($n=4096$,38.42) ($n=8192$,32.50) ($n=16384$,28.14)  };
			\addplot+[postaction={
				pattern=north west lines
			}] coordinates { ($n=4096$,42.81) ($n=8192$,54.37) ($n=16384$,59.57)  };
			\addplot+[postaction={
				pattern=dots
			}] coordinates { ($n=4096$,2.27) ($n=8192$,4.04) ($n=16384$,5.60)  };
			\addplot+[postaction={
				pattern=crosshatch
			}] coordinates { ($n=4096$,13.44) ($n=8192$,5.53) ($n=16384$,3.05)  };
			
			\legend{Dense,Low-rank,RHS+Sol,Spectra}
			\end{axis}
			\end{tikzpicture}			
			\caption{Distribution of the time spent in the different tasks in 
			  Algorithm~\ref{alg:dac2d}. The results are for some instances 
			  of the 2D Fractional Laplace equation considered in
			  Section~\ref{sec:2d-laplace-test} with fADI as low-rank solver
			  and $n_{\min} = 2048$.}
			\label{fig:hist_2D_Fractional_Laplace}
	\end{figure}
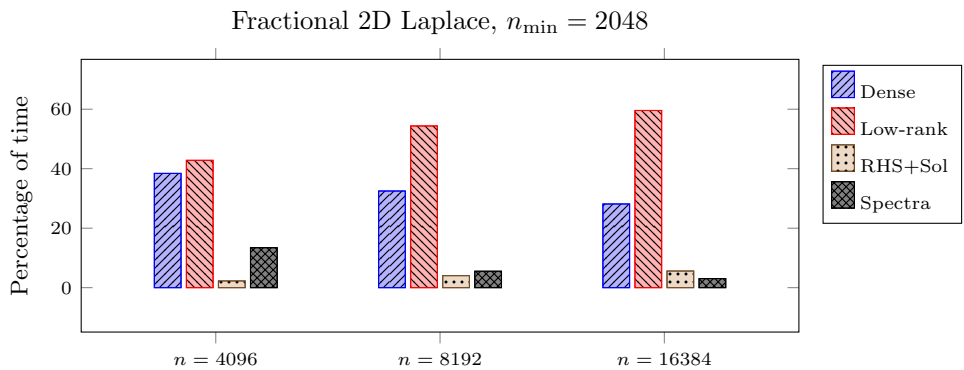

	\subsection{Varying the block size}
	We perform numerical tests similar to the ones of the
	previous section, but we only consider the low-rank solver 
	fADI for the update equations and, instead, we vary the minimal 
	block size, aiming at determining the best block-size for each 
	test problem. Table~\ref{tab:2d-laplace}
	and Table~\ref{tab:2d-fractional} report the results concerning 
	$n_{\min}\in\{256, 512, 1024\}$ for the 2D Laplace equation, and 
	$n_{\min}\in\{2048, 4096, 8192\}$ for the fractional 
	Gr\"unwald-Letnikov finite differences case. 
	
	The results indicate that the choice $n_{\min} = 512$ is ideal for the 
	2D Laplace case, and $n_{\min} = 4096$ for the fractional case. 
	We expect that problems involving larger off-diagonal ranks will need 
	a larger choice of $n_{\min}$.

	\begin{table}	
		\caption{Timings and residuals for the solution of the 2D 
		Laplace equation of size $n \times n$ by means of Algorithm~\ref{alg:dac2d}
		using fADI as low rank solver, with different choices of 
		$n_{\min}$.}
		\resizebox{\textwidth}{!}{ 
			\begin{filecontents*}{test2_2D.dat}
%% LaTeX2e file `test2_2D.dat'
%% generated by the `filecontents' environment
%% from source `paper' on 2023/01/17.
%%
  2048 1.0882 5.2158e-13 1.9669 3.6599e-10 0.98486 3.1171e-10 0.9854 1.789e-10
  4096 5.8803 1.3431e-12 4.9768 3.8945e-10 3.5436 3.3681e-10 3.6807 2.0215e-10
  8192 42.044 3.2897e-12 19.642 4.1506e-10 14.74 3.6683e-10 16.504 2.4277e-10
  16384 319.9 7.8397e-12 80.409 4.1732e-10 59.833 3.7037e-10 67.058 2.5e-10
  32768 2512.5 8.8771e-12 340.49 4.1909e-10 259.14 3.7272e-10 283.31 2.5534e-10 
 
\end{filecontents*}
\pgfplotstabletypeset[skip rows between index={0}{0},
			empty cells with={\ensuremath{-}},
			every head row/.style={
				before row={
					\toprule
					\multicolumn{1}{c|}{}&
					\multicolumn{2}{c|}{\textbf{diag}}&
					\multicolumn{2}{c|}{\textbf{dc\_adi}}&
					\multicolumn{2}{c|}{\textbf{dc\_adi}}&
					\multicolumn{2}{c}{\textbf{dc\_adi}} \\
					\multicolumn{1}{c|}{}&
					\multicolumn{2}{c|}{}&
					\multicolumn{2}{c|}{$n_{\min} = 256$}&
					\multicolumn{2}{c|}{$n_{\min} = 512$}&
					\multicolumn{2}{c}{$n_{\min} = 1024$}
					\\
				},
				after row = \midrule,
			},
			columns/0/.style = {column name = $n$, column type=c|},
			columns/1/.style = {column name = Time,precision=1,zerofill, fixed},
			columns/2/.style = {column name = Res,precision=1,zerofill, column type=c|},
			columns/3/.style = {column name = Time,precision=1,zerofill, fixed},
			columns/4/.style = {column name = Res,precision=1,zerofill, column type=c|},
			columns/5/.style = {column name = Time,precision=1,zerofill, fixed},
			columns/6/.style = {column name = Res,precision=1,zerofill, column type=c|},
			columns/7/.style = {column name = Time,precision=1,zerofill, fixed},
			columns/8/.style = {column name = Res,precision=1,zerofill, column type=c},
			]{test2_2D.dat}}
		\label{tab:2d-laplace-nmin}
	\end{table}

	\begin{table}
		\caption{Timings and residuals for the solution of the 2D 	  
		  fractional Laplace equation of size $n \times n$ by means of Algorithm~\ref{alg:dac2d}
		  using fADI as low rank solver, with different choices of 
		  $n_{\min}$.}
		\resizebox{\textwidth}{!}{ 
			\begin{filecontents*}{test3_2D.dat}
%% LaTeX2e file `test3_2D.dat'
%% generated by the `filecontents' environment
%% from source `paper' on 2023/01/17.
%%
  2048 2.1748 6.3687e-15 NaN NaN NaN NaN NaN NaN
  4096 16.802 8.1799e-15 21.826 5.353e-11 NaN NaN NaN NaN
  8192 139.29 1.7944e-14 108.87 1.2156e-10 100.64 5.509e-11 NaN NaN
  16384 1036.1 2.3063e-14 524.6 1.4613e-10 478.93 8.5513e-11 652.59 5.4519e-11
  32768 7934.2 3.1935e-14 2337.5 1.7103e-10 2034.1 1.1976e-10 2741.5 8.1899e-11 
 
\end{filecontents*}
\pgfplotstabletypeset[skip rows between index={0}{0},
			empty cells with={\ensuremath{-}},
			every head row/.style={
				before row={
					\toprule
					\multicolumn{1}{c|}{}&
					\multicolumn{2}{c|}{\textbf{diag}}&
					\multicolumn{2}{c|}{\textbf{dc\_adi}}&
					\multicolumn{2}{c|}{\textbf{dc\_adi}}&
					\multicolumn{2}{c}{\textbf{dc\_adi}} \\
					\multicolumn{1}{c|}{}&
					\multicolumn{2}{c|}{}&
					\multicolumn{2}{c|}{$n_{\min} = 2048$}&
					\multicolumn{2}{c|}{$n_{\min} = 4096$}&
					\multicolumn{2}{c}{$n_{\min} = 8192$}
					\\
				},
				after row = \midrule,
			},
			columns/0/.style = {column name = $n$, column type=c|},
			columns/1/.style = {column name = Time,precision=1,zerofill, fixed},
			columns/2/.style = {column name = Res,precision=1,zerofill, column type=c|},
			columns/3/.style = {column name = Time,precision=1,zerofill, fixed},
			columns/4/.style = {column name = Res,precision=1,zerofill, column type=c|},
			columns/5/.style = {column name = Time,precision=1,zerofill, fixed},
			columns/6/.style = {column name = Res,precision=1,zerofill, column type=c|},
			columns/7/.style = {column name = Time,precision=1,zerofill, fixed},
			columns/8/.style = {column name = Res,precision=1,zerofill, column type=c},
			]{test3_2D.dat}}
		\label{tab:2d-fractional-nmin}
	\end{table}
	
	\subsection{Comparison with 2D state-of-the art solvers}
	\label{sec:superdc}
	
	In this section we compare Algorithm~\ref{alg:dac2d} with solvers based 
	on fast diagonalization strategies (\textbf{dst} and \textbf{superdc}). The fast diagonalization 
	procedure of \textbf{superdc} requires to set a 
	minimal block size. We have chosen $n_{\min} = 2048$, which yields the best 
	results on the cases of study. 
	
	The results are shown in Table~\ref{tab:2d-laplace-superdc} and \ref{tab:2d-lapfrac-superdc} and the running times are also displayed in Figure~\ref{fig:running-times}.
	Both \textbf{dst} and \textbf{superdc} have the quasi-optimal complexity $\mathcal O(n^2 \log n)$, 
	as Algorithm~\ref{alg:dac2d}. Algorithm~\ref{alg:dac2d} significantly 
	outperforms \textbf{superdc} in all examples, with an acceleration of more than 
	10x on the largest examples. The performances in the 2D Laplace case are 
	comparable with those of \textbf{dst}. We remark that, having as size a power of $2$ is the worst case scenario for the performance of the discrete sine transform; we expect an additional speed up of the approach based on \textbf{dst} when considering sizes of the form $2^k-1$. On the other hand, \textbf{dst} only 
	applies to the specific case of the 2D Laplace equation with constant 
	coefficients.

	\begin{table}
		\caption{Timings and residuals for the solution of the 2D 
		Laplace equation of size $n \times n$ by means of Algorithm~\ref{alg:dac2d}
		using fADI and $n_{\min} = 512$, \textbf{superdc} and \textbf{dst}.}
		\centering
		\begin{filecontents*}{test7.dat}
%% LaTeX2e file `test7.dat'
%% generated by the `filecontents' environment
%% from source `paper' on 2023/01/17.
%%
 1024 0.57544 2.4926e-10 0.225640 6.345830e-14 nan nan 
 2048 1.0375 3.1349e-10 0.853807 1.144327e-13 nan nan
 4096 3.5119 3.4297e-10 3.270838 3.271964e-13 15.51 2.9332e-12
 8192 14.924 3.6532e-10 13.497730 1.110817e-12 95.311 4.5167e-11
 16384 61.409 3.7013e-10 55.122466 2.684789e-12 593.36 9.5937e-11
 32768 258.06 3.7227e-10 240.604803 2.353214e-12 3342.6 1.5423e-10
 
\end{filecontents*}
\pgfplotstabletypeset[skip rows between index={0}{0},
		every head row/.style={
			before row={
				\toprule
				\multicolumn{1}{c|}{}&
				\multicolumn{2}{c|}{\textbf{dc\_adi}}&
	\multicolumn{2}{c|}{\textbf{dst}}&
	\multicolumn{2}{c}{\textbf{superdc}}
				\\
				\multicolumn{1}{c|}{}&
				\multicolumn{2}{c|}{$n_{\min} = 512$}&
				\multicolumn{2}{c|}{}&
				\multicolumn{2}{c}{$n_{\min} = 2048$}
				\\
			},
			after row = \midrule,
		},
		columns/0/.style = {column name = $n$, column type=c|},
			columns/1/.style = {column name = Time,precision=1,zerofill, fixed},
		columns/2/.style = {column name = Res,precision=1,zerofill, column type=c|},
			columns/3/.style = {column name = Time,precision=1,zerofill, fixed},
	columns/4/.style = {column name = Res,precision=1,zerofill, column type=c|},
			columns/5/.style = {column name = Time,precision=1,zerofill, fixed},
	columns/6/.style = {column name = Res,precision=1,zerofill, column type=c},
		]{test7.dat}
		\label{tab:2d-laplace-superdc}
	\end{table}

		\begin{table}
			\caption{Timings and residuals for the solution of the 2D 
			Fractional Laplace equation of size $n \times n$ by means of Algorithm~\ref{alg:dac2d}
			using fADI and $n_{\min} = 2048, 4096$, and \textbf{superdc}.}
			\centering
			\begin{filecontents*}{test7-1.dat}
%% LaTeX2e file `test7-1.dat'
%% generated by the `filecontents' environment
%% from source `paper' on 2023/01/17.
%%
  4096 21.826 5.353e-11 NaN NaN 91.405 4.0442e-09
  8192 108.87 1.2156e-10 100.64 5.509e-11 685.72 5.3405e-09
  16384 524.6 1.4613e-10 478.93 8.5513e-11 4715 5.7253e-09
  32768 2337.5 1.7103e-10 2034.1 1.1976e-10 27546 5.8792e-09
 
\end{filecontents*}
\pgfplotstabletypeset[skip rows between index={0}{0},
			every head row/.style={
				before row={
					\toprule
					\multicolumn{1}{c|}{}&
					\multicolumn{2}{c|}{\textbf{dc\_adi}}&
		\multicolumn{2}{c|}{\textbf{dc\_adi}}&
		\multicolumn{2}{c}{\textbf{superdc}}
					\\
					\multicolumn{1}{c|}{}&
					\multicolumn{2}{c|}{$n_{\min} = 2048$}&
					\multicolumn{2}{c|}{$n_{\min} = 4096$}&
					\multicolumn{2}{c}{$n_{\min} = 2048$}
					\\
				},
				after row = \midrule,
			},
			columns/0/.style = {column name = $n$, column type=c|},
				columns/1/.style = {column name = Time,precision=1,zerofill, fixed},
			columns/2/.style = {column name = Res,precision=1,zerofill, column type=c|},
				columns/3/.style = {column name = Time,precision=1,zerofill, fixed},
		columns/4/.style = {column name = Res,precision=1,zerofill, column type=c|},
				columns/5/.style = {column name = Time,precision=1,zerofill, fixed},
		columns/6/.style = {column name = Res,precision=1,zerofill, column type=c},
			]{test7-1.dat}
			\label{tab:2d-lapfrac-superdc}
		\end{table}
	
	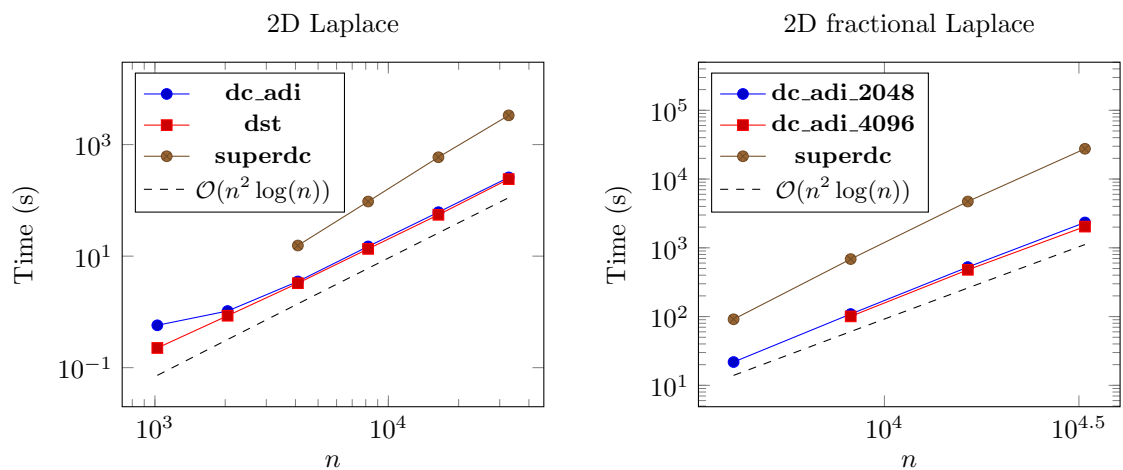
\begin{figure}
		\begin{tikzpicture}
			\begin{loglogaxis}[legend pos = north west, width = .485\linewidth, legend style={font=\small}, ymax=3e4, ylabel = {Time (s)}, xlabel = $n$, title = {2D Laplace}]
\addplot table[x index = 0, y index = 1]{test7.dat};
\addplot table[x index = 0, y index = 3]{test7.dat};
\addplot table[x index = 0, y index = 5]{test7.dat};
\addplot[domain = 1024:32768, dashed]{1e-8*x^2 * ln(x)};
\legend{\textbf{dc\_adi}, \textbf{dst}, \textbf{superdc}, $\mathcal O(n^2\log(n))$};
				\end{loglogaxis}
		\end{tikzpicture}~~~~~~\
			\begin{tikzpicture}
		\begin{loglogaxis}[legend pos = north west, width = .485\linewidth, legend style={font=\small}, ymax=5e5, ylabel = {Time (s)}, xlabel = $n$, title = {2D fractional Laplace}]
			\addplot table[x index = 0, y index = 1]{test7-1.dat};
			\addplot table[x index = 0, y index = 3]{test7-1.dat};
			\addplot table[x index = 0, y index = 5]{test7-1.dat};
			\addplot[domain = 4096:32768, dashed]{1e-7*x^2 * ln(x)};
			\legend{\textbf{dc\_adi\_2048}, \textbf{dc\_adi\_4096}, \textbf{superdc}, $\mathcal O(n^2\log(n))$};
		\end{loglogaxis}
	\end{tikzpicture}
\caption{Log-log plot of the running times for the solution of the 2D Laplace equation (left) and the 2D fractional Laplace equation (right) associated with the numerical tests in Section~\ref{sec:superdc}.}\label{fig:running-times}
	\end{figure}
	
	\subsection{3D Laplace equation}
	
	We now test the 3D version of the Laplace solver described in 
	Section~\ref{sec:2d-laplace-test}. More precisely, we solve 
	the tensor equation 
	\begin{equation} \label{eq:3d-test}
		\T X \times_1 A_1 + \T X \times_2 A_2 + \T X \times_3 A_3 = \T B, 
	\end{equation}
	where $A_t$ are finite difference discretization of the 1D Laplacian
	with zero Dirichlet boundary condition of sizes $n_i \times n_i$. 
	The reference solution $\T X$ is randomly generated by means 
	of \texttt{randn(n1,n2,n3)}, and $\T B$ is set evaluating \eqref{eq:3d-test}. 
	We remark that in the 3D case we can choose two different block sizes: one 
	for the recursion in the tensor equation, and one for the recursive 
	calls to the 2D solver. In this section we indicate the former 
	with $n_{\min}$ and the latter is set to $1024$ in all the examples, 
	with the only exception of the scaling test in Section~\ref{sec:3d-scaling}, 
	where all block sizes (2D and 3D) are set to $32$.

	The low-rank solver for the update equations is fADI, 
	and the tolerance $\epsilon$ in Algorithm~\ref{alg:dac}
	is set to $\epsilon = 10^{-6}$.
	We consider two test cases.
	
	\paragraph{Test 1}
	We choose $n = n_1 = n_2 = n_3$ ranging in 
	$\{ 256, 512, 1024 \}$  and the considered block sizes
	are $n_{\min} \in \{ 128, 256, 512 \}$. Note that, in this 
	test case all the 2D problems in the recursion are solved with 
	the dense method, in view of our choice of the minimal block size. 
	The results are reported in Table~\ref{tab:3d-laplace-nmin}. 
	
	The results show that the dense method is faster for all choices 
	of $n$ and $n_{\min}$, although the scaling 
	suggests that a breakeven point should be reached around
	$n=2048$ and $n_{\min}$ $1024$. However, this is not achievable with 
	the computational resources at our disposal, since in the case 
	$n = 2048$ the solution cannot be stored in the system memory.

	\begin{table}
		\caption{Timings and residuals for the solution of the 3D 
		  Laplace equation of dimension $n \times n \times n$ 
		  by means of Algorithm~\ref{alg:dac2d}
		  using fADI as low rank solver, with different choices of 
		  $n_{\min}$ for the 3D splitting. The $n_{\min}$ used in the recursive 
		  2D solver is fixed to $n_{\min} = 1024$.}
		\resizebox{\textwidth}{!}{ 
			\begin{filecontents*}{test4_3D.dat}
%% LaTeX2e file `test4_3D.dat'
%% generated by the `filecontents' environment
%% from source `paper' on 2023/01/17.
%%
 256 1.0351 9.857e-15 4.8024 1.6281e-08 NaN NaN NaN NaN
 512 11.52 1.1286e-14 25.356 2.1031e-08 18.123 1.3272e-08 NaN NaN
 1024 144.12 1.8032e-14 258.01 2.3509e-08 242.47 1.6936e-08 193.43 1.052e-08
 
\end{filecontents*}
\pgfplotstabletypeset[skip rows between index={0}{0},
			empty cells with={\ensuremath{-}},
			every head row/.style={
				before row={
					\toprule
					\multicolumn{1}{c|}{}&
					\multicolumn{2}{c|}{\textbf{diag}}&
					\multicolumn{2}{c|}{\textbf{dc\_adi}}&
					\multicolumn{2}{c|}{\textbf{dc\_adi}}&
					\multicolumn{2}{c}{\textbf{dc\_adi}} \\
					\multicolumn{1}{c|}{}&
					\multicolumn{2}{c|}{}&
					\multicolumn{2}{c|}{$n_{\min} = 128$}&
					\multicolumn{2}{c|}{$n_{\min} = 256$}&
					\multicolumn{2}{c}{$n_{\min} = 512$}
					\\
				},
				after row = \midrule,
			},
			columns/0/.style = {column name = $n$, column type=c|},
			columns/1/.style = {column name = Time,precision=1,zerofill, fixed},
			columns/2/.style = {column name = Res,precision=1,zerofill, column type=c|},
			columns/3/.style = {column name = Time,precision=1,zerofill, fixed},
			columns/4/.style = {column name = Res,precision=1,zerofill, column type=c|},
			columns/5/.style = {column name = Time,precision=1,zerofill, fixed},
			columns/6/.style = {column name = Res,precision=1,zerofill, column type=c|},
			columns/7/.style = {column name = Time,precision=1,zerofill, fixed},
			columns/8/.style = {column name = Res,precision=1,zerofill, column type=c},
			]{test4_3D.dat}}
		\label{tab:3d-laplace-nmin}
	\end{table}
	
	\paragraph{Test 2}
	
	We choose the unbalanced dimensions $n_1 \times 512 \times 512$ with 
	$n_1 = 2^j$ for $j = 10, \ldots, 14$. We choose $n_{\min} = 256$ 
	for the 3D splitting. Since the recursion is structured to split 
	larger dimensions first, also in this case the recursive 2D problems 
	are solved with the dense solver. The results in Table~\ref{tab:3d-laplace-n1}
	confirm the expected linear scaling with respect to $n_1$, and 
	the approach is faster than the dense solver from dimension 
	$n_1 = 4096$.

		\begin{table}
			\caption{Timings and residuals for the solution of the 3D 
			  Laplace equation of dimension $n_1 \times 512 \times 512$ 
			  by means of Algorithm~\ref{alg:dac2d}
			  using fADI as low rank solver, with 
			  $n_{\min} = 256$ for the 3D splitting, and $n_{\min} = 1024$ for the 
			  recursive 2D solver.}
			\centering
				\begin{filecontents*}{test5_3D.dat}
%% LaTeX2e file `test5_3D.dat'
%% generated by the `filecontents' environment
%% from source `paper' on 2023/01/17.
%%
  1024 25.674 1.2733e-14 40.951 1.3719e-08
  2048 68.523 1.4607e-14 87.925 1.4077e-08
  4096 201.95 1.436e-14 187.54 1.422e-08
  8192 680.83 1.523e-14 404.03 1.4287e-08
  16384 2417.9 1.6778e-14 975.12 1.4319e-08 
  
\end{filecontents*}
\pgfplotstabletypeset[skip rows between index={0}{0},
				empty cells with={\ensuremath{-}},
				every head row/.style={
					before row={
						\toprule
						\multicolumn{1}{c|}{}&
						\multicolumn{2}{c|}{\textbf{diag}}&
						\multicolumn{2}{c}{\textbf{dc\_adi}} \\
						\multicolumn{1}{c|}{}&
						\multicolumn{2}{c|}{}&
						\multicolumn{2}{c}{$n_{\min} = 256$}
						\\
					},
					after row = \midrule,
				},
				columns/0/.style = {column name = $n_1$, column type=c|},
				columns/1/.style = {column name = Time,precision=1,zerofill, fixed},
				columns/2/.style = {column name = Res,precision=1,zerofill, column type=c|},
				columns/3/.style = {column name = Time,precision=1,zerofill, fixed},
				columns/4/.style = {column name = Res,precision=1,zerofill, column type=c},
				]{test5_3D.dat}
			\label{tab:3d-laplace-n1}
		\end{table}
	
		\subsection{Asymptotic complexity in the 3D case}
		\label{sec:3d-scaling}
	
		The previous experiment 
		provides too few data points to assess the expected cubic complexity.
		In addition, the use of the dense solver does not allow to validate the error 
		analysis that we have performed, and that 
		guarantees that the inexact solving of the subproblems does not
		destroy the final accuracy. 
		
		To validate 
		the scaling and accuracy of Algorithm~\ref{alg:dac} as $n$ grows, 
		we set the minimal block size 
		to the small value $n_{\min} = 32$, and we measure the timings 
		for problems of size between $n = 64$ and $n = 1024$.
		The results are reported in Figure~\ref{fig:3d-scaling} and  Figure~\ref{fig:running-times3D}; the latter 
		confirm the predicted accuracy and the almost cubic scaling. 
	
		In addition, in the right part of Figure~\ref{fig:3d-scaling} we 
		display the time spent in the various parts of Algorithm~\ref{alg:dac}. 
		This highlights that the solution of the update equations dominates the 
		other costs, which is expected in view of the small $n_{\min}$. We 
		remark that the latter
		include the calls to the 2D solver described in Algorithm~\ref{alg:dac2d}.

		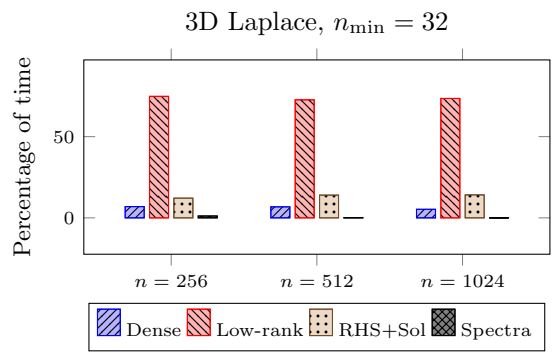
\begin{figure}
			\centering
			\begin{minipage}{.35\linewidth}
				\begin{filecontents*}{test6_3D.dat}
%% LaTeX2e file `test6_3D.dat'
%% generated by the `filecontents' environment
%% from source `paper' on 2023/01/17.
%%
 64 0.55549 2.2607e-08
 128 1.7255 2.9292e-08
 256 11.526 3.3761e-08
 512 89.08 3.6305e-08
 1024 839.79 7.7682e-08  
  
\end{filecontents*}
\pgfplotstabletypeset[skip rows between index={0}{0},
				empty cells with={\ensuremath{-}},
				every head row/.style={
					before row={
						\toprule
						\multicolumn{1}{c|}{}&
						\multicolumn{2}{c}{\textbf{dc\_adi}} \\
						\multicolumn{1}{c|}{}&
						\multicolumn{2}{c}{$n_{\min} = 32$}
						\\
					},
					after row = \midrule,
				},
				columns/0/.style = {column name = $n$, column type=c|},
				columns/1/.style = {column name = Time,precision=1,zerofill, fixed},
				columns/2/.style = {column name = Res,precision=1,zerofill, column type=c},
				]{test6_3D.dat}
			\end{minipage}~\begin{minipage}{.64\linewidth}
				\begin{tikzpicture}
					\begin{axis}[
					title = \text{3D Laplace, $n_{\min} = 32$},
					x tick label style={
						/pgf/number format/1000 sep=},
					ylabel=\small{Percentage of time},
					width = .82\linewidth, height = .2\textheight,
					symbolic x coords={$n=256$, $n=512$, $n=1024$},
					xtick = data,
					tick label style={font=\scriptsize},
					enlargelimits=0.3,
					ybar,
					bar width = .25cm,
					legend style={					
						font=\scriptsize,
						at={(0.99,-0.25)},
						legend columns = -1,
						cells={anchor=west}
					},
					legend image code/.code={
						\draw [#1] (0.2cm,-0.05cm) rectangle (0.5cm,0.3cm); },
					],
					]
					\addplot+[postaction={
						pattern=north east lines
					}] coordinates { ($n=256$,6.93) ($n=512$,6.79) ($n=1024$,5.27)  };
					\addplot+[postaction={
						pattern=north west lines
					}] coordinates { ($n=256$,74.80) ($n=512$,72.68) ($n=1024$,73.52)  };
					\addplot+[postaction={
						pattern=dots
					}] coordinates { ($n=256$,12.13) ($n=512$,14.00) ($n=1024$,14.15)  };
					\addplot+[postaction={
						pattern=crosshatch
					}] coordinates { ($n=256$,1.16) ($n=512$,0.20) ($n=1024$,0.04)  };
					\legend{Dense,Low-rank,RHS+Sol,Spectra}
					\end{axis}
					\end{tikzpicture}					
			\end{minipage}
			\caption{On the left, timings and residuals for the 3D Laplace example 
			  in Section~\ref{sec:3d-scaling}, with $n_{\min} = 32$ and fADI 
			  as a low-rank solver for the update equations. On the right, the 
			  distribution of the time spent in the different subtasks of 
			  Algorithm~\ref{alg:dac}. }
			\label{fig:3d-scaling}
		\end{figure}
	
		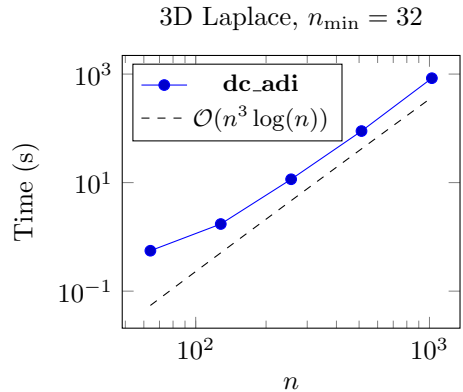
\begin{figure}
			\centering
		\begin{tikzpicture}
			\begin{loglogaxis}[legend pos = north west, height = .25\textheight, legend style={font=\small}, ylabel = {Time (s)}, xlabel = $n$, title = {3D Laplace, $n_{\min}=32$}]
				\addplot table[x index = 0, y index = 1]{test6_3D.dat};
				\addplot[domain = 64:1024, dashed]{5e-8*x^3 * ln(x)};
				\legend{\textbf{dc\_adi},  $\mathcal O(n^3\log(n))$};
			\end{loglogaxis}
		\end{tikzpicture}
	\caption{Log-log plot of the running times for the solution of the 3D Laplace example in Section~\ref{sec:3d-scaling}.}\label{fig:running-times3D}
	\end{figure}
		\section{Conclusions}
	We have proposed a new solver for positive definite tensor Sylvester equation with hierarchically low-rank 
	coefficients that attains the quasi-optimal complexity $\mathcal O(n^d\log(n))$. Our procedure is based on a 
	nested divide-and-conquer paradigm. We have developed an error analysis that reveals the relation between the level
	 on inexactness in the solution of the nested subproblems and the final accuracy. 
	
	 The numerical results demonstrate that the proposed solver can significantly speed up the solution of 
	 matrix Sylvester equations of medium size. In the 3D-tensor case with equal size,  
	the method is slower than the dense solver based on diagonalization, when addressing sizes
	 up to $1024\times1024\times1024$. On the other hand, the performances are quite close, and we expect
	  that running the simulations in a distributed memory environment or on a machine with a high level
	   of performance would uncover a breakeven point around $2048\times2048\times2048$. 
	
	A further speed up might be reached employing a relaxation strategy for the inexactness 
	of the linear system solving in fADI or RK \cite{kurschner2020inexact}. This may provide significant 
	advantages for $d>2$, where the time spent on solving the equations with low-rank right-hand side is above 
	the $70\%$ of the total.
	
	Another promising direction is to adapt the method to exploit block low-rank structures in the
	 right-hand side; this will subject of future investigations.

\bibliographystyle{siam}
\bibliography{library}

\end{document}